\documentclass[preprint, times]{elsarticle}
\usepackage{amsmath}
\usepackage{waveletnotation}
\usepackage{epsfig}

\def\eop{\hfill\rule{2.0mm}{2.0mm}}

\newcommand{\diag}{\mathrm{diag}}

\newtheorem{lemma}{Lemma}

\newtheorem{cor}[lemma]{Corollary}
\newtheorem{theorem}[lemma]{Theorem}
\newtheorem{example}{Example}
\newtheorem{algorithm}{Algorithm}
\newproof{proof}{Proof}
\begin{document}

\title{Construction of Symmetric Complex Tight $\df$-wavelet Frames from Pseudo Splines via Matrix Extension with Symmetry}

\author[rvt]{Xiaosheng Zhuang\corref{cor1}\fnref{fn1}}
\ead{xzhuang@math.ualberta.ca}
\address[rvt]{Department of Mathematical and Statistical Sciences,
University of Alberta, Edmonton, Alberta, Canada T6G 2G1. }

\cortext[cor1]{Corresponding author, {\tt
http://www.ualberta.ca/$\sim$xzhuang}}

\fntext[fn1]{Research was supported in part by NSERC Canada under
Grant RGP 228051.}

\makeatletter \@addtoreset{equation}{section} \makeatother

\begin{abstract}
\end{abstract}



\maketitle
\pagenumbering{arabic}

\section{Introduction and Motivation}
Redundant wavelet systems have been proved to be quit useful in many
applications, for examples, signal denoising, image processing,
numerical algorithm, and so on. As a redundant system, it can
possess many desirable properties such as symmetry, short support,
high vanishing moments, and so on, simultaneously. In this paper, we
are interested in the construction of tight wavelet frames with such
desirable properties. Other than redundancy, computational
efficiency is also desired in applications, which means a wavelet
system should have as few generators as possible. It is well known
that the construction of a wavelet system from a refinable function
can be formulated as a matrix extension problem.  In order to have
fewer generators for a tight wavelet frame, we shall employ the
algorithm for matrix extension with symmetry.

Let us first introduce some notation and definitions. We say that
$\df$ is a \emph{dilation factor} if $\df>1$ is an integer.
Throughout this paper, $\df$ is fixed and denotes a dilation factor.
We say that $\phi:\R\mapsto \C$ is a \emph{$\df$-refinable function}
if
\begin{equation}\label{refeq}
\phi=\df\sum_{k\in\Z} a_0(k)\phi(\df \cdot-k),
\end{equation}
where $a_0:\Z\mapsto \C$ is a finitely supported sequence on $\Z$,
called the \emph{low-pass filter} for $\phi$. The \emph{symbol} of
$a_0$ is  given by $\pa_0(z):=\sum_{k\in\Z}a_0(k)z^k$. The
\emph{coefficient support} of $a_0$ is defined by $\cs(a_0):=[m,n]$,
where $a_0(m)a_0(n)\neq0$ and $a_0(k)=0$ for all $k\notin[m,n]$. In
frequency domain, the refinement equation in \eqref{refeq} can be
rewritten as
\begin{equation}\label{refeq:freq}
\hat\phi(\df \xi) = \wh{a_0}(\xi)\hat\phi(\xi), \quad \xi\in\R,
\end{equation}
where $\wh{a_0}$ is the \emph{Fourier series} of $a_0$ given by
\begin{equation}\label{Fseq}
\wh{a_0}(\xi):=\sum_{k\in\Z}a_0(k)e^{-\iu k\xi}=\pa_0(e^{-i\xi}),
\quad \xi\in\R.
\end{equation}
The Fourier transform $\hat{f}$ of $f\in L_1(\R)$ is defined to be
$\hat{f}(\xi)=\int_{\R} f(t) e^{-\iu t\xi} dt$ and can be extended
to square integrable functions and tempered distributions. A
low-pass filter $a_0$ is \emph{orthogonal} if
$\sum_{j=0}^{\df-1}|\wh{a_0}(\xi+2\pi j/\df)|^2 = 1$ for all
$\xi\in\R$.

Usually, a wavelet system is generated by some wavelet function
$\psi^\ell, \ell=1,\ldots, L$ from a $\df$-refinable function $\phi$
as follows:
\begin{equation}\label{waveletGenerators}
\wh{\psi^\ell}(\df\xi) =\wh a_\ell(\xi)\hat\phi(\xi), \quad
\ell=1,\ldots,L,
\end{equation}
where each $a_\ell:\Z\mapsto\C$ is a finitely supported sequence on
$\Z$, called the \emph{high-pass filter} for $\psi^\ell$, $
\ell=1,\ldots, L$.

We says that $\{\psi^1,\ldots,\psi^L\}$ generates a
\emph{$\df$-wavelet frame} in $L_2(\R)$ if
$\{\psi^\ell_{j,k}:=\df^{j/2}\psi^\ell(\df^j\cdot-k)\,:\, j,k\in\Z,
\ell=1,\ldots,L\}$ is a frame in $L_2(\R)$, that is, there exists
two positive constants $C_1, C_2$ such that
\begin{equation}
\label{def:frame} C_1\|f\|_{L_2(\R)}^2\le
\sum_{\ell=1}^L\sum_{j\in\Z}\sum_{k\in\Z} |\la
f,\psi^\ell_{j,k}\ra|^2\le C_2\|f\|_{L_2(\R)}^2\quad \forall f\in
L_2(\R),
\end{equation}
where $|\la f,\psi^\ell_{j,k}\ra|^2=\la f,\psi^\ell_{j,k}\ra\la
\psi^\ell_{j,k},f\ra$ and $\la\cdot,\cdot\ra$ is the inner product
in $L_2(\R)$ defined by
\[
\la f,g\ra=\int_\R f(t)\overline{g(t)}dt, \quad f,g\in L_2(\R).
\]
If $C_1=C_2=1$ in \eqref{def:frame}, we say that
$\{\psi^1,\ldots,\psi^L\}$ generates a \emph{tight $\df$-wavelet
frame} in $L_2(\R)$.

An important property of a wavelet system is its order of vanishing
moments. We say that $\{\psi^1,\ldots,\psi^L\}$ has \emph{vanishing
moments} of order $n$ if
\begin{equation}\label{def:vanishingMoment}
\int_\R t^k\psi^\ell(t)dt=0\qquad k=0,\ldots,n-1, \forall
\ell=1,\ldots,L.
\end{equation}
\eqref{def:vanishingMoment} is equivalent to saying that
$\frac{d^k}{dt^k}\wh\psi^\ell(0)=0$ for all $k=0,\ldots,n-1$ and
$\ell=1,\ldots,L$.

Let  $\phi$ be a compactly supported $\df$-refinable function in
$L_2(\R)$ associated with a low-pass filter $a_0$ such that
$\hat\phi(0)=1$. Suppose there exist high-pass filters
$a_1,\ldots,a_L$ such that
\begin{equation}\label{UEP}
\sum_{\ell = 0}^L\wh{a_\ell}\overline{\wh{a_\ell}(\cdot+2\pi
k/\df)}=\delta_k, \quad k=0,\ldots, \df-1,
\end{equation}
where $\delta$ is the \emph{Dirac sequence} such that $\delta(0)=1$
and $\delta(k) = 0 $ for all $k\neq 0$. Define $\psi^\ell, \ell =
1,\ldots,L$ as in \eqref{waveletGenerators}. Then
$\{\psi^1,\ldots,\psi^L\}$ generates a tight $\df$-wavelet frame,
see \cite{Shen, RonShen1}. Moreover, if the low-pass filter $a_0$
satisfies
\begin{equation}\label{eq:lowerpassat0}
 1-|\wh{a_0}(\xi)|^2 = O(|\xi|^{2n}),
\end{equation}
which means $1-|\wh{a_0}(\xi)|^2$ has \emph{zero of order} $2n$ near
the origin, then $\{\psi^1,\ldots,\psi^L\}$ given by
\eqref{waveletGenerators} has vanishing moments of order $n$, see
\cite{DaubHanRonShen}.

Now, we are ready to introduce our main result of this paper on the
construction of symmetric complex tight $\df$-wavelet frame.
Throughout this paper, we shall denote $P_{m,n}(y)$ a polynomial of
degree $n-1$ as follows:
\begin{equation}\label{def:poly}
P_{m,n}(y)=
\sum_{j=0}^{n-1}\left[\sum_{j_1+\cdots+j_{\df-1}=j}\prod_{k=1}^{\df-1}{m-1+j_k
\choose j_k} \sin(k\pi/\df)^{-2j_k}\right]y^j.
\end{equation}
\begin{theorem}
\label{thm:main:1} Let $\df$ be a dilation factor. Let $m, n\in\N$
be positive integers such that $2n-1\le m$. Let $P_{m,n}(y)$ be the
polynomial defined in \eqref{def:poly}. Then
\begin{equation}\label{eq:polyGreater0}
P_{m,2n-1}(y)>0 \qquad \forall y\in\R.
\end{equation}
Let
$z_1,\overline{z_1},\ldots,z_{n-1},\overline{z_{n-1}}\in\C\setminus\R$
be all the complex roots of $P_{m,2n-1}$. Then
\begin{equation}\label{eq:Qmn}
P_{m,2n-1}(y) = |Q_{m,n}(y)|^2,
\end{equation}
where $Q_{m,n}(y) = c(y-z_1)\cdots(y-z_{n-1})$ with
$c=(-1)^{n-1}(z_1\cdots z_{n-1})^{-1}$. Define a low-pass filter
$a_0$ by
\begin{equation}\label{def:Amn}
\wh{a_0}(\xi):= e^{\iu\lfloor \frac{m(\df-1)}{2}\rfloor\xi}
\left(\frac{1+e^{-\iu \xi}+\cdots+e^{-\iu (\df-1)\xi}}{\df}\right)^m
Q_{m,n}(\sin^2(\xi/2)),
\end{equation}
where $\lfloor \cdot\rfloor$ is the floor operation. Then
\begin{equation}\label{eq:symAmn}
\wh{a_0}(-\xi)= e^{\iu (1-\gep)\xi} \wh{a_0}(\xi)\quad \mbox
{with}\quad  \gep = m(\df-1)-2\lfloor \frac{m(\df-1)}{2}\rfloor,
\end{equation}
and
\[
\cs(a_0) = [-\lfloor \frac{m(\df-1)}{2}\rfloor-n+1,\lfloor
\frac{m(\df-1)}{2}\rfloor+n-1+\gep].
\]
Let $\phi$ be the standard $\df$-refinable function associated with
the low-pass filter $a_0$, that is,
$\wh{\phi}(\xi):=\prod_{j=1}^\infty\wh{a_0}(\df^{-j} \xi)$. Then
$\phi$ is a compactly supported $\df$-refinable function in
$L_2(\R)$ with symmetry $\phi(\frac{1-\gep}{\df-1}-\cdot)=\phi$.
Moreover, one can construct high-pass filters $a_1, \ldots, a_{L}$,
$L\in\{\df-1,\df, \df+1\}$, with symmetry by Algorithm \ref{alg:2}
such that \eqref{UEP} holds. Define $\psi^1,\ldots,\psi^L$ as in
\eqref{waveletGenerators}. Then $\{\psi^1,\ldots,\psi^L\}$ generates
a tight $\df$-wavelet frame in $L_2(\R)$ and
$\{\psi^1,\ldots,\psi^L\}$ has vanishing moments of order $2n-1$.
Furthermore, $|\cs(a_\ell)|\le |\cs(a_0)|$  and
$\psi^\ell(\frac{1-\gep}{\df-1}+\eta^\ell-\cdot)=\gep^\ell\psi^\ell$
for some $\eta^\ell\in\{1,0,-1\}$ and $\gep^\ell\in\{-1,1\}$, $\ell
= 1,\ldots,L$. In particular, for $m=2n-1$, $L=\df-1$ and
$\{\psi^1,\ldots,\psi^{\df-1}\}$ generates an orthonormal wavelet
basis with symmetry for $L_2(\R)$.
\end{theorem}

 Theorem~\ref{thm:main:1} is summarized from the following sections.
 In section~2, we shall inspect the
important properties of $P_{m,n}(y)$ and show that for any dilation
factor $\df$ and $m,n\in\N$, we can construct a symmetric complex
$\df$-refinable function from $P_{m,2n-1}(y)$,  called
\emph{{\textbf{complex}} $\df$-refinable pseudo spline of type I
with order $(m,2n-1)$}. Moreover, we prove that the shifts of all
such complex pseudo splines are linearly independent. In section~3,
we introduce the general problem of matrix extension with symmetry
and present a step-by-step algorithm for the matrix extension with
symmetry. In section~4, we shall apply our matrix extension
algorithm and show that we can construct symmetric complex tight
$\df$-wavelet frames from symmetric complex pseudo spline with only
$\df$ wavelet generators. Moreover, we shall provide several
examples to demonstrate our results and algorithms. Conclusions will
be given in last section.

\section{Symmetric Complex Refinable Functions from Pseudo Splines}
In this section, we shall restrict our attention to the construction
of $\df$-refinable functions from which wavelet system with high
vanishing moments can be derived. In other words, we need to design
low-pass filters $a$ such that $a$ satisfies \eqref{eq:lowerpassat0}
for some $n\in\N$.

To guarantee that the $\df$-refinable function $\phi$ associated
with $a$ has certain regularity, usually the low-pass filter $a$
satisfies the \emph{sum rules of order} $m$ for some $m\in\N$:
\begin{equation}
(1+e^{-\iu\xi}+\cdots+e^{-\iu (\df-1)\xi})^m \mid \hat{a}(\xi).
\end{equation}
That is, $\wh{a}$ is of the form:
\begin{equation}\label{eq:maskA}
\hat{a}(\xi) =
e^{\iu\lfloor\frac{m(\df-1)}{2}\rfloor}\left(\frac{1+e^{-\iu\xi}+\cdots+e^{-\iu
(\df-1)\xi}}{\df}\right)^m \wh{L}(\xi),
\end{equation}
for some $2\pi$-periodic trigonometric polynomial $\wh L$ with $\wh
L(0)=1$. For $\wh L(\xi) \equiv 1$.  $\wh{a}(\xi)$ is the low-pass
filter for B-spline  of order $m$: $\wh{B_m}(\xi) = e^{\iu\lfloor
m/2\rfloor\xi}(1-e^{-\iu\xi})^m/(\iu\xi)^m$.

Define a function $h$ by
\begin{equation}\label{eq:h}
h(y):=\prod_{k=1}^{\df-1}\left(1-\frac{y}{\sin^2(k\pi/\df)}\right),
\quad y\in\R.
\end{equation}
We have
\begin{equation}\label{eq:hInv}
h(\sin^2(\xi/2))=\frac{|1+\cdots+e^{-\iu
(\df-1)\xi}|^2}{\df^2}=\frac{\sin^2(\df\xi/2)}{\df^2\sin^2(\xi/2)}
\end{equation}
and
\begin{equation}\label{eq:taylorh}
h(y)^{-m} =
\left[\prod_{k=1}^{\df-1}\left(\sum_{j_k=0}^\infty\frac{y^{j_k}}{\sin^{2j_k}(k\pi/\df)}\right)\right]^{-m}=\sum_{j=0}^\infty
c_{m,j}y^j, \quad |y|<\sin^2(\pi/\df),
\end{equation}
where
\begin{equation}\label{eq:coefh}
c_{m,j}=\sum_{j_1+\cdots+j_{\df-1}=j}\prod_{k=1}^{\df-1}{m-1+j_k
\choose j_k} \sin(k\pi/\df)^{-2j_k}, \quad j\in\N.
\end{equation}
Note that $P_{m,n}(y)=\sum_{j=0}^{n-1}c_{m,j}y^j$. Consequently, we
have the following result.
\begin{lemma}\label{lemma:PolyPmn}
Let $m,n\in\N$ be such that $n\le m$; let $P_{m,n}$ and $h$ be
polynomials defined as in \eqref{def:poly} and  \eqref{eq:h},
respectively. Then $P_{m,n}(\sin^2(\xi/2))$ is the unique positive
trigonometric polynomial of minimal degree such that
\begin{equation}\label{eq:derivingP}
1-h(\sin^2(\xi/2))^mP_{m,n}(\sin^2(\xi/2)) = O(|\xi|^{2n}).
\end{equation}
\end{lemma}
\begin{proof} Suppose there exist trigonometric polynomial $\hat
g(\xi)$ such that
\[
1-h(\sin^2(\xi/2))^m g(\xi) =O(|\xi|^{2n}).
\]
Observing that the first $(n-1)th$-degree Taylor polynomial of
$h(y)^{-m}$ at $y=0$ is $P_{m,n}(y)$, we have
\[
g(\xi) = h(\sin^2(\xi/2))^{-m}[1+O(|\xi|^{2n})] =
P_{m,n}(\sin^2(\xi/2))+O(|\sin^2(\xi/2)|^n).
\]
Moreover, it is easily seen that the coefficients of $P_{m,n}(y)$
are all positive. Consequently, $P_{m,n}(\sin^2(\xi/2))$ is   the
unique positive trigonometric polynomial of minimal degree such that
\eqref{eq:derivingP} holds. \eop
\end{proof}

For  $m,n\in\N$ such that $n\le m$, let $\wh{_{II}
a}(\xi):=h(\sin^2(\xi/2)^mP_{m,n}(\sin^2(\xi/2))$. Then the
$\df$-refinable function $_{II}\phi$ associated with ${_{II} a}$ by
\eqref{refeq} is called the \emph{pseudo spline of type II}. By
Lemma \ref{lemma:PolyPmn}, using Riesz Lemma,  one can  derive a
low-pass filter $_{I}a$ from $_{II}a$ such that $|\wh{_I a}(\xi)|^2
=\wh{_{II} a}(\xi)$. The $\df$-refinable function $_I\phi$
associated with such $_I a$ by \eqref{refeq} is referred as
\emph{real pseudo spline of type I}. Interesting readers can refer
to \cite{DongShen1,DongShen1,DongShen1,DongShen1} for more details
on this subject for the special case $\df = 2$.

Note that $_I a$ satisfies \eqref{eq:lowerpassat0}. One can
construct high-pass filters $b^1,\ldots,b^L$ from $_Ia$ such that
\eqref{UEP} holds. Then $\psi^1,\ldots,\psi^L$ defined by
\eqref{waveletGenerators} are real-valued functions.
$\{\psi^1,\ldots,\psi^L\}$ has vanishing moment of order $n$ and
generates a tight $\df$-wavelet frame. However,
$\{\psi^1,\ldots,\psi^L\}$ is not necessary symmetric since the
low-pass filter $_Ia$ from $_{II}a$ via Riesz lemma might not posses
any symmetry pattern. In what follows, we shall consider
complex-valued wavelet generators and show that we can achieve
symmetry for any odd integer $n\in N$.

The next lemma is needed later, which generalizes \cite[Lemma
2.1]{LiMoYu}.

\begin{lemma}\label{lemma:derivativePoly}
Let $P_{m,n}(y)$ be defined as in \eqref{def:poly}. Then
\begin{equation}\label{eq:derivativePoly}
P_{m,n}'(y) =
m\sum_{\ell=1}^{\df-1}\frac{1}{(1-\frac{y}{\sin^2(\frac{\ell\pi}{\df})})}\left[
\frac{P_{m,n}(y)}{\sin^2(\frac{\ell\pi}{\df})}-c^\ell_{n-1}y^{n-1}\right],
\end{equation}
where
\[
c^\ell_{n-1}=\sum_{j_1+\cdots+j_{\df-1}=n-1}\left[{m+j_\ell \choose
j_\ell}\sin(\frac{\ell\pi}{\df})^{-2(j_\ell+1)}\prod_{k=1,k\neq
\ell}^{\df-1}{m-1+j_k \choose
j_k}\sin(\frac{k\pi}{\df})^{-2j_k}\right].
\]
\end{lemma}
\begin{proof}
By \eqref{def:poly}, we have
\[
\begin{aligned}
P_{m,n}'(y) &=
\sum_{j=0}^{n-2}\left[\sum_{j_1+\cdots+j_{\df-1}=j+1}\prod_{k=1}^{\df-1}{m-1+j_k
\choose j_k} \sin(\frac{k\pi}{\df})^{-2j_k}\right](j+1)y^j\\
& =m
\sum_{j=0}^{n-2}\left[\sum_{\ell=1}^{\df-1}\sum_{j_1+\cdots+j_{\df-1}=j+1}\frac{j_\ell}{m}\prod_{k=1}^{\df-1}{m-1+j_k
\choose j_k} \sin(\frac{k\pi}{\df})^{-2j_k}\right]y^j\\
&=m
\sum_{\ell=1}^{\df-1}\sum_{j=0}^{n-2}\left[\sum_{j_1+\cdots+j_{\df-1}=j}{m+j_\ell
\choose j_\ell}
\sin(\frac{\ell\pi}{\df})^{-2(j_\ell+1)}\prod_{k=1,k\neq\ell}^{\df-1}{m-1+j_k
\choose j_k} \sin(\frac{k\pi}{\df})^{-2j_k}\right]y^j
\\
&=:m\sum_{\ell=1}^{\df-1}Q_\ell(y),
\end{aligned}
\]
where
\[
Q_\ell(y):=\sum_{j=0}^{n-2}\left[\sum_{j_1+\cdots+j_{\df-1}=j}{m+j_\ell
\choose j_\ell}
\sin(\frac{\ell\pi}{\df})^{-2(j_\ell+1)}\prod_{k=1,k\neq\ell}^{\df-1}{m-1+j_k
\choose j_k} \sin(\frac{k\pi}{\df})^{-2j_k}\right]y^j
\]
Note that
\[
\begin{aligned}
&\frac{y}{\sin^2(\frac{\ell\pi}{\df})}Q_\ell(y)
\\&=\sum_{j=0}^{n-2}\left[\sum_{j_1+\cdots+j_{\df-1}=j}{m+j_\ell
\choose j_\ell}
\sin(\frac{\ell\pi}{\df})^{-2(j_\ell+2)}\prod_{k=1,k\neq\ell}^{\df-1}{m-1+j_k
\choose j_k} \sin(\frac{k\pi}{\df})^{-2j_k}\right]y^{j+1}
\\&=\sum_{j=1}^{n-1}\left[\sum_{j_1+\cdots+j_{\df-1}=j}{m-1+j_\ell
\choose j_\ell-1} \sin(\frac{\ell\pi}{\df})^{-2(j_\ell+
1)}\prod_{k=1,k\neq\ell}^{\df-1}{m-1+j_k \choose j_k}
\sin(\frac{k\pi}{\df})^{-2j_k}\right]y^{j}.
\end{aligned}
\]
Using ${m+j_\ell \choose j_\ell}-{m-1+j_\ell \choose
j_\ell-1}={m-1+j_\ell \choose j_\ell}$, it is easy to deduce that
\[
\left(1-\frac{y}{\sin^2(\frac{\ell\pi}{\df})}\right)Q_\ell(y)=\frac{P_{m,n}(y)}{\sin^2(\frac{\ell\pi}{\df})}-c^\ell_{n-1}y^{n-1},
\]
where
\[
c^\ell_{n-1}=\sum_{j_1+\cdots+j_{\df-1}=n-1}\left[{m+j_\ell \choose
j_\ell}\sin(\frac{\ell\pi}{\df})^{-2(j_\ell+1)}\prod_{k=1,k\neq
\ell}^{\df-1}{m-1+j_k \choose
j_k}\sin(\frac{k\pi}{\df})^{-2j_k}\right].
\]
Consequently, \eqref{eq:derivativePoly} holds.
 \eop
\end{proof}

By the above lemma, we have the following result regarding the
positiveness of $P_{m,n}(y)$, which generalizes \cite[Theorem
5]{Han} and \cite[Theorem 2.4]{LiHanYu} and proves
\eqref{eq:polyGreater0}.

\begin{theorem}
Let $m,n\in\N$ be  such that $n\le m$. Then $P_{m,n}(y)>0$ for all
$y\in\R$ if and only if $n$ is an odd number.
\end{theorem}
\begin{proof}
It is easily seen that $P_{m,n}(y)>0$ for all $y\ge0$ due to the
positiveness of its coefficients and
$\lim_{y\rightarrow-\infty}P_{m,n}(y)=-\infty$ when $n$ is even.
Hence $n$ must be an odd number if $P_{m,n}(y)>0$ for all $y\in\R$.

On the other hand, suppose $n$ is an odd number. Let $y_0<0$ be any
stationary point of $P_{m,n}(y)$. Then by Lemma
\ref{lemma:derivativePoly}, we have
\[
0=P'_{m,n}(y_0)=m\sum_{\ell=1}^{\df-1}\frac{1}{(1-\frac{y_0}{\sin^2(\frac{\ell\pi}{\df})})}\left[
\frac{P_{m,n}(y_0)}{\sin^2(\frac{\ell\pi}{\df})}-c^\ell_{n-1}y_0^{n-1}\right],
\]
which implies that $P_{m,n}(y_0)>0$.  Consequently, $P_{m,n}(y)>0$
for all $y<0$. This completes our proof. \eop
\end{proof}

Now, by $P_{m,2n-1}(y)>0$ for all $y\in\R$ and $2n-1\le m$,
$P_{m,2n-1}(y)$ can only have complex roots. Hence, we must have
\[
P_{m,2n-1}(y)=c_0\prod_{j=1}^{n-1}(y-z_j)(y-\overline{z_j}),\quad
z_1,\overline{z_1},\ldots,z_{n-1},\overline{z_{n-1}}\in\C\setminus\R.
\]
Define $Q_{m,n}(y) = c\prod_{j=1}^{n-1}(y-z_j)$ as in \eqref{eq:Qmn}
and $\wh{a_0}(\xi)$ as in \eqref{def:Amn}. Then it is easy to check
that the symmetry pattern of $a_0$  satisfies \eqref{eq:symAmn}. We
shall refer the $\df$-refinable function $\phi_{m,n}$ associated
with the low-pass filter $a_0$ as \emph{complex pseudo spline of
type I}.

Now, we have the following result which shall play an important role
in our construction of tight $\df$-wavelet frame in section~3.
\begin{cor}\label{cor:maskRemainder}
 Let $m,n\in\N$ be such that $2n-1\le m$ and $a_0$ be
defined as in \eqref{def:Amn}. Then
\begin{equation}\label{eq:maskForUEP}
1-\sum_{j=0}^{\df-1}|\wh{a_0}(\xi+2\pi j/\df)|^2 =|\wh{
b_0}(\df\xi)|^2,
\end{equation}
for some $2\pi$-periodic trigonometric function $\wh {b_0}(\xi)$
with real coefficients. In particular, $\wh
{b_0^{2n-1,n}}(\xi)\equiv0$ and $|\wh {b_0^{2n,n}}(\xi)|^2 =
c_{2n,2n-1}[\sin^2(\xi/2)/\df^2]^{2n-1}$ with $c_{2n,2n-1}$ being
the coefficient given in \eqref{eq:coefh}.
\end{cor}

\begin{proof} We first show the $1-\sum_{j=0}^{\df-1}|\wh{a_0}(\xi+2\pi
j/\df)|^2\ge 0$ for all $\xi\in\R$. Let $y_j:=\sin^2(\xi/2+\pi
j/\df)$, $j=0,\ldots,\df-1$. Since $|\wh{a_0}(\xi)|^2 =
h(\sin^2(\xi/2))^mP_{m,2n-1}(\sin^2(\xi/2))$, we have
\[
\begin{aligned}
1-\sum_{j=0}^{\df-1}|\wh{a_0}(\xi+2\pi
j/\df)|^2&=1-\sum_{j=0}^{\df-1}h(y_j)^mP_{m,2n-1}(y_j)
\\&=1-\sum_{j=0}^{\df-1}h(y_j)^mP_{m,m}(y_j)+\sum_{j=0}^{\df-1}h(y_j)^m\sum_{k=2n-1}^{m-1}c_{m,k}y_j^k
\\&=\sum_{j=0}^{\df-1}h(y_j)^m\sum_{k=2n-1}^{m-1}c_{m,k}y_j^k\ge 0.
\end{aligned}
\]
The last equality follows from  the fact that the low-pass filter
$a$ satisfying $|\hat{a}(\xi)|^2=h(y_0)^mP_{m,m}(y_0)$  is an
orthogonal low-pass filter , which satisfies
$\sum_{j=0}^{\df-1}h(y_j)^mP_{m,m}(y_j)=1$ (\cite{RankM}). Now, by
that $1-\sum_{j=0}^{\df-1}|\wh{a_0}(\xi+2\pi j/\df)|^2\ge 0$ is of
period $2\pi/\df$, \eqref{eq:maskForUEP} follows from Riesz Lemma.

Obviously, $\wh{b_0^{2n-1,n}}(\xi)\equiv0$ since $a_0^{2n-1,n}$ is
an orthogonal low-pass filter. And from above, noting that
$h(y_j)y_j=\sin^2(\df\xi/2)/\df^2$. we have
\[
\begin{aligned}
|\wh{b_0^{2n,2n-1}}(\df\xi)|^2
&=c_{2n,2n-1}\sum_{j=0}^{\df-1}h(y_j)^{2n}y_j^{2n-1}
=c_{2n,2n-1}\sum_{j=0}^{\df-1}(h(y_j)y_j)^{2n-1})h(y_j)
\\&=c_{2n,2n-1}[\sin^2(\df\xi/2)/\df^2]^{2n-1}\sum_{j=0}^{\df-1}h(y_j)P_{1,1}(y_j)
\\&=c_{2n,2n-1}[\sin^2(\df\xi/2)/\df^2]^{2n-1}.
\end{aligned}
\]
\eop
\end{proof}

Next, we shall discuss the linear independence of the pseudo spline
of type I (real and complex) and type II. Linear independence is an
important issue in both approximation theory and wavelet analysis.
The linear independence of  the integer shifts of a refinable
function  is a necessary and sufficient condition for the existence
of a compactly supported dual refinable function, see \cite{7. P. G.
LemariRieusset On the existence of compactly supported dual
waveletsAppl. Comput. Harmon. Anal. 3 (1997) 117{118}.}. For a
compactly supported function $\phi\in L_2(\R)$, we say that the
shifts of $\phi$ are \emph{linearly independent} if
\begin{equation}\label{def:linear.Independent}
\sum_{j\in\Z}c(j)\phi(\cdot-j)=0\quad \mbox{implies} \quad c(j)=0\,
\forall j\in\Z.
\end{equation}
We say that the shifts of $\phi$ are \emph{stable} if there exists
two positive constants $C_1$ and $C_2$ such that
\begin{equation}\label{def:stable}
C_1\sum_{j\in\Z}|c(j)|^2 \le
\|\sum_{j\in\Z}c(j)\phi(\cdot-j)\|_{L_2(\R)}^2\le
C_2\sum_{j\in\Z}|c(j)|^2,
\end{equation}
for all finitely supported sequences $c:\Z\mapsto\C$.

It is known in \cite{JiaMichelli} that the shifts of a compacly
supported function $\phi\in L_2(\R)$ is linearly independent (or
stable) if and only if
\begin{equation}\label{eq:lin.ind.and.stable}
\mathrm{span}\{ \hat\phi(\xi+2\pi k)\,:\, k\in\Z\}\neq 0\,\, \mbox{
for all }\,\, \xi\in\C\,(\mbox{or for all}\,\, \xi\in\R).
\end{equation}
Here, the
\emph{Fourier--Lapalce transform} of $\phi$  is defined to be
\begin{equation}\label{def:fourierLaplace}
\hat{\phi}(\xi) = \int_\R\phi(t)e^{-\iu \xi t} dt, \quad \xi\in\C.
\end{equation}
Consequently, if the shifts of a compactly supported function
$\phi\in L_2(\R)$ are linearly independent, then the shifts of
$\phi$ must be stable in $L_2(\R)$.

The following lemma is needed later to prove the linear independence
of the shifts of pseudo spline of type I or type II.
\begin{lemma}\label{lemma:increaseCmk}
Let $m,n\in\N$ be such that $1<n<m$. Let $c_{m,j}$ be the
coefficient of $P_{m,n}(y)$ defined in \eqref{eq:coefh}. Then
\begin{equation}\label{eq:increaseCmk}
2c_{m,j-1}<c_{m,j},\quad  j=1,\ldots n-1.
\end{equation}
\end{lemma}
\begin{proof} First, it is easy to show that $2{m-1+j \choose j}<{m+j \choose
j+1}$ for $j=0,\ldots, n-2$. Then
\[
\begin{aligned}
2c_{m,j-1}&=2\sum_{j_1+\cdots+j_{\df-1}=j-1}\prod_{k=1}^{\df-1}{m-1+j_k
\choose j_k} \sin(\frac{k\pi}{\df})^{-2j_k}
\\&=
\frac{1}{\df-1}\sum_{\ell=1}^{\df-1}\sum_{j_1+\cdots+j_{\df-1}=j-1}2{m-1+j_\ell
\choose j_\ell}
\sin(\frac{\ell\pi}{\df})^{-2j_\ell}\prod_{k=1,k\neq\ell}^{\df-1}{m-1+j_k
\choose j_k} \sin(\frac{k\pi}{\df})^{-2j_k}
\\&<
\frac{1}{\df-1}\sum_{\ell=1}^{\df-1}\sum_{j_1+\cdots+j_{\df-1}=j-1}{m+j_\ell
\choose j_\ell+1}
\sin(\frac{\ell\pi}{\df})^{-2(j_\ell+1)}\prod_{k=1,k\neq\ell}^{\df-1}{m-1+j_k
\choose j_k} \sin(\frac{k\pi}{\df})^{-2j_k}
\\&=
\frac{1}{\df-1}\sum_{\ell=1}^{\df-1}\sum_{j_1+\cdots+j_{\df-1}=j}{m-1+j_\ell
\choose j_\ell}
\sin(\frac{\ell\pi}{\df})^{-2(j_\ell)}\prod_{k=1,k\neq\ell}^{\df-1}{m-1+j_k
\choose j_k} \sin(\frac{k\pi}{\df})^{-2j_k}
\\& =
\frac{1}{\df-1}\sum_{\ell=1}^{\df-1}\sum_{j_1+\cdots+j_{\df-1}=j}\prod_{k=1}^{\df-1}{m-1+j_k
\choose j_k} \sin(\frac{k\pi}{\df})^{-2j_k}
\\& = c_{m,j}.
\end{aligned}
\]
 \eop
\end{proof}

Due to
$\wh{_{II}\phi}(\xi)=\wh{_I\phi}(\xi)\cdot\wh{\overline{_I\phi}}(-\xi)$
for all $\xi\in\C$, the shifts of the pseudo spline of type II are
linearly independent will implies the shifts of the pseudo spline of
type I are linearly independent as well \cite[Proposition
1.1]{DongShen3}. Moreover, we have the following lemma, whose proof
is also similar to \cite[Lemma 2.1]{DongShen3}

\begin{lemma}\label{lemma:eqivalent.Lin.Ind.}
Let $\phi\in L_2(\R)$ be a compactly supported refinable function
associated with a low-pass filter $a$ by \eqref{refeq}. Then, the
shifts of $\phi$ are linearly independent if and only if
\begin{itemize}
\item[{\rm (1)}] the shifts of $\phi$ are stable;

\item[{\rm (2)}] the symbol $\pa$ of $a$ satisfies: $\{\pa(z e^{\iu 2\pi j/\df}): j=0,\ldots,\df-1\}\neq
0$ for all $z\in\C$.
\end{itemize}
\end{lemma}

For a polynomial $P(z)=c_0+c_1z+\cdots+c_nz^n$ with real
coefficients satisfying $c_n>c_{n-1}>\cdots>c_0>0$, it was shown in
\cite{DongShen3} that all zeros of $P(z)$ are containing inside the
open unit disk $\{z\in\C\,:\,|z|<1\}$.

Using the above facts, we can prove the following result.

\begin{theorem}\label{thm:lin.ind.pseduo.spline}
The shifts of the pseudo spline of type I (real or complex) or type
II are linearly independent.
\end{theorem}
\begin{proof}We only need to show that the shifts of the pseudo
spline of type II are linearly independent.

Fix $m,n\in\N$ with $1\le n \le m$. Let the low-pass filter $a$ be
given as follows:
\[
\hat{a}(\xi) = h(\sin^2(\xi/2))^mP_{m,n}(\sin^2(\xi/2)).
\]
Let  $\phi$ be the pseudo spline of type II associated with mask $a$
by \eqref{refeq}.

For $n=m$, $\phi$ associated with $a$  is an  interpolatory
$\df$-refinable function \cite{HanKwangZhuang}, i.e.,
$\phi(j)=\delta(j), j\in\Z$ and $\hat{\phi}(\df\xi)=\hat
a(\xi)\hat\phi(\xi)$. In this case, the shifts of $\phi$ must be
linearly independent.

For $n=1$, by that the shifts of the B-spline of order $2m$ are
stable and the symbol of $a$ is
$\pa(z)=|1+\cdots+z^{\df-1}|^{2m}/\df^{2m}$, the fact that the
shifts of $\phi$ associated with $a$ are linearly independent
follows by Lemma \ref{lemma:eqivalent.Lin.Ind.}.

 Hence, without loss of
generality, we can assume that $1<n<m$. Since $\hat{a}(\xi) \ge
h(\sin^2(\xi/2))^m$ and the shifts of B-spine of order $2m$ are
stable, by \eqref{eq:lin.ind.and.stable}, the shifts of $\phi$ are
stable.  We need to prove that the symbol of $a$ satisfies item (2)
of Lemma \ref{lemma:eqivalent.Lin.Ind.}. By
$\sin^2(\xi/2)=-\frac{(1-z)^2}{4z}$ with $z=e^{-\iu \xi}$, we have
\[
\begin{aligned}
\pa(z) &=
\left|\frac{1+z+\cdots+z^{\df-1}}{\df}\right|^{2m}P_{m,n}\left(-\frac{(1-z)^2}{4z}\right)
\\&=\left|\frac{1+z+\cdots+z^{\df-1}}{\df}\right|^{2m}\sum_{j=0}^{n-1}c_{m,j}\left[-\frac{(1-z)^2}{4z}\right]^j
\\&
=\left|\frac{1+z+\cdots+z^{\df-1}}{\df}\right|^{2m}\sum_{j=0}^{n-1}2^{-j}c_{m,j}\left[-\frac{(1-z)^2}{2z}\right]^j.
\end{aligned}
\]
Now, item (2) of Lemma \ref{lemma:eqivalent.Lin.Ind.} is equivalent
to saying that there is no $z\in\C$ such that $P(ze^{\iu 2\pi
j/\df}) = 0$ for all $j=0,\ldots, \df-1$, where $P(z) =
\sum_{j=0}^{n-1}2^{-j}c_{m,j}\left[-\frac{(1-z)^2}{2z}\right]^j$.
Suppose not, let $z_0=\rho e^{\iu\theta}\in \C$ be such that
$P(z_j)=0$ for all $j=0,\ldots,\df-1$, where $z_j=\rho
e^{\iu\theta_j}$ with $\theta_j = \theta+2\pi j/\df$ for
$j=0,\ldots,\df-1$. Note that
\[
\left|-\frac{(1-z_j)^2}{2z_j}\right| =
\frac{|1-z_j|^2}{2\rho}=\frac{\rho^2+1}{2\rho}-\cos(\theta_j)\ge
1-\cos(\theta_j), \quad j=0,\ldots,\df-1.
\]
Consequently, $\max_{0\le j\le
\df-1}\left|-\frac{(1-z_j)^2}{2z_j}\right|\ge 1$. This contradicts
to the fact that all roots of the polynomial
$\sum_{j=0}^{n-1}2^{-j}c_{m,j} z^j$ are containing inside the unit
disk $\{z\in\C\,:\, |z|<1\}$ due to its coefficients $2^jc_{m,j},
j=0,\ldots, n-1$ satisfying $2^{-j+1}c_{m,j-1}<2^{-j}c_{m,j}$ for
all $j=1,\ldots,n-1$. We are done.
 \eop
\end{proof}

\section{Matrix Extension with Symmetry}
In this section, we shall introduce the matrix extension problem,
which plays an important role in our construction of tight
$\df$-wavelet frame. In order to state the matrix extension problem
and our main results, let us introduce some notation and definitions
first.

 Let $\pp(z)=\sum_{k\in \Z} p_k z^k,
z\in \C \bs \{0\}$ be a Laurent polynomial with complex coefficients
$p_k\in \C$ for all $k\in \Z$. We say that $\pp$ has {\it symmetry}
if its coefficient sequence $\{p_k\}_{k\in \Z}$ has symmetry; more
precisely, there exist $\gep \in \{-1,1\}$ and $c\in \Z$ such that
\begin{equation}\label{sym:seq}
p_{c-k}=\gep p_k, \qquad \forall\; k\in \Z.
\end{equation}
If $\gep=1$, then $\pp$ is symmetric about the point $c/2$; if
$\gep=-1$, then $\pp$ is antisymmetric about the point $c/2$.
Symmetry of a Laurent polynomial can be conveniently expressed using
a symmetry operator $\sym$ defined by
\begin{equation}\label{sym}
\sym \pp(z):=\frac{\pp(z)}{\pp(1/z)}, \qquad z\in \C \bs \{0\}.
\end{equation}
When $\pp$ is not identically zero,  it is evident that
\eqref{sym:seq} holds if and only if $\sym \pp(z)=\gep z^c$. For the
zero polynomial, it is very natural that $\sym 0$ can be assigned
 any symmetry pattern; that is, for every occurrence
of $\sym 0$ appearing in an identity in this paper, $\sym 0$  is
understood to take an appropriate choice of $\gep z^c$ for some
$\gep\in \{-1,1\}$ and $c\in \Z$ so that the identity holds. If
$\pP$ is an $r\times s$ matrix of Laurent polynomials with symmetry,
then we can apply the operator $\sym$ to each entry of $\pP$, that
is, $\sym \pP$ is an $r\times s$ matrix such that $[\sym
\pP]_{j,k}:=\sym ([\pP]_{j,k})$, where $[\pP]_{j,k}$ denotes the
$(j,k)$-entry of the matrix $\pP$ throughout the chapter.

For two matrices $\pP$ and $\pQ$ of Laurent polynomials with
symmetry, even though all the entries in $\pP$ and $\pQ$ have
symmetry, their sum $\pP+\pQ$, difference $\pP-\pQ$, or product
$\pP\pQ$, if well defined, generally may not have symmetry any more.
This is one of the difficulties for matrix extension with symmetry.
In order for $\pP\pm\pQ$ or $\pP \pQ$ to possess some symmetry, the
symmetry patterns of $\pP$ and $\pQ$ should be compatible. For
example, if $\sym \pP=\sym \pQ$, that is, both $\pP$ and $\pQ$ have
the same symmetry pattern, then indeed $\pP\pm\pQ$ has symmetry and
$\sym(\pP\pm\pQ)=\sym \pP=\sym \pQ$.  In the following, we discuss
the compatibility of symmetry patterns of matrices of Laurent
polynomials. For an $r\times s$ matrix $\pP(z)=\sum_{k\in \Z} P_k
z^k$,  we denote
\begin{equation}\label{Pstar}
\pP^*(z):=\sum_{k\in \Z} P_k^* z^{-k} \quad \mbox{with}
\quad P_k^*:=\ol{P_k}^T,\qquad k\in \Z,
\end{equation}
where $\ol{P_k}^T$ denotes the transpose of the complex conjugate of
the constant matrix $P_k$ in $\C$. We say that \emph{the symmetry of
$\pP$ is compatible} or \emph{$\pP$ has compatible symmetry}, if
\begin{equation}\label{sym:comp}
\sym \pP(z)=(\sym \pth_1)^*(z) \sym \pth_2(z),
\end{equation}
for some $1 \times r$ and $1\times s$ row vectors $\pth_1$ and
$\pth_2$ of Laurent polynomials with symmetry. For an $r\times s$
matrix $\pP$ and an $s\times t$ matrix $\pQ$ of Laurent polynomials,
we say that $(\pP, \pQ)$  \emph{has mutually compatible symmetry} if
%
\begin{equation}\label{sym:mutual}
\sym \pP(z)=(\sym \pth_1)^*(z) \sym \pth(z) \quad \mbox{and} \quad
\sym \pQ(z)=(\sym \pth)^*(z) \sym \pth_2(z)
\end{equation}
for some $1\times r$, $1\times s$,  $1\times t$ row vectors $\pth_1,
\pth, \pth_2$ of Laurent polynomials with symmetry. If $(\pP,\pQ)$
has mutually compatible symmetry as in \eqref{sym:mutual}, then it
is easy to verify that their product $\pP \pQ$ has compatible
symmetry and in fact $\sym (\pP\pQ)= (\sym \pth_1)^*\sym \pth_2 $.

For a matrix of Laurent polynomials, another important property is
the support of its coefficient sequence. For $\pP=\sum_{k\in \Z} P_k
z^k$ such that $P_k={\bf0}$ for all $k\in \Z \bs [m,n]$ with
$P_{m}\ne{\bf0}$ and $P_{n}\ne {\bf0}$, we define its coefficient
support to be $\cs(\pP):=[m,n]$ and the length of its coefficient
support to be $|\cs(\pP)|:=n-m$. In particular, we define
$\cs({\bf0}):=\emptyset$, the empty set, and
$|\cs({\bf0})|:=-\infty$. Also, we  use $\coeff(\pP,k):=P_k$ to
denote the coefficient matrix (vector) $P_k$ of $z^k$ in $\pP$. In
this chapter, ${\bf0}$ always denotes a general zero matrix whose
size can be determined in the context.

The Laurent polynomials that we shall consider in this paper have
their coefficients in a subfield $\F$ of the complex field $\C$. Let
$\F$ denote a subfield of $\C$ such that $\F$ is closed under the
operations of complex conjugate of $\F$ and square roots of positive
numbers in $\F$. In other words, the subfield $\F$ of $\C$ satisfies
the following properties:
\begin{equation}\label{F}
\bar{x}\in \F \quad \hbox{and}\quad \sqrt{y}\in \F, \qquad \forall\;
x,y \in \F\quad \mbox{with}\quad y>0.
\end{equation}
Two particular examples of such subfields $\F$ are $\F=\R$ (the
field of real numbers) and $\F=\C$ (the field of complex numbers). A
nontrivial example other than $\R$ and $\C$ that satisfies \eqref{F}
is the field of all algebraic number, i.e., the algebraic closure
$\overline{\Q}$ of the rational number $\Q$. A subfield of $\R$
given by $\overline{\Q}\cap\R$ also satisfies \eqref{F}.

Now, we introduce the general matrix extension problem with
symmetry. Let  $r$ and $s$ denote two positive integers such that
$1\le r\le s$. Let $\pP$ be an $r\times s$ matrix of Laurent
polynomials with coefficients in $\F$ such that $\pP(z)
\pP^*(z)=I_{r}$ for all $z\in \C \bs \{0\}$ and the symmetry of
$\pP$ is compatible, where $I_r$ denotes the $r\times r$ identity
matrix. The matrix extension problem with symmetry is to find an
$s\times s$ square matrix $\pP_e$ of Laurent polynomials with
coefficients in $\F$ and with symmetry such that $[I_r, \mathbf{0}]
\pP_e=\pP$ (that is, the submatrix of the first $r$ rows of $\pP_e$
is the given matrix $\pP$), the symmetry of $\pP_e$ is compatible,
and $\pP_e(z)\pP_e^*(z)=I_{s}$ for all $z\in \C \bs \{0\}$ (that is,
$\pP_e$ is paraunitary).  Moreover, in many applications, it is
often highly desirable that the coefficient support of $\pP_e$ can
be controlled by that of $\pP$ in some way.


In the context of wavelet analysis, matrix extension without
symmetry has been discussed by Lawton, Lee and Shen in their
interesting paper \cite{LLS} and a simple algorithm has been
proposed there to derive a desired matrix $\pP_e$ from a given row
vector $\pP$ of Laurent polynomials without symmetry. In electronic
engineering, an algorithm using the cascade structure for matrix
extension without symmetry has been given in \cite{V} for filter
banks with perfect reconstruction property. The algorithms in
\cite{LLS,V} mainly deal with the special case that $\pP$ is a row
vector without symmetry and the coefficient support of the derived
matrix $\pP_e$ indeed can be controlled by that of $\pP$.  For
$\F=\R$ and $r=1$, matrix extension with symmetry has been
considered in \cite{P}.

We study the general matrix extension problem with symmetry and have
the following results.

\begin{theorem}\label{thm:main:2}
Let $\pP$ be an $r\times s$ matrix of Laurent polynomials with
coefficients in a subfield $\F$ of $\C$ such that \eqref{F} holds.
Then $\pP(z) \pP^*(z)=I_r$ for all $z\in \C \bs \{0\}$
and the symmetry of $\pP$ is compatible as in \eqref{sym:comp}, if
and only if, there exist $s\times s$ matrices $\pP_0, \ldots,
\pP_{J+1}$ of Laurent polynomials with coefficients in $\F$ such
that
\begin{enumerate}
\item[{\rm(1)}] $\pP_e:= \pP_{J+1}(z) \pP_J(z)\cdots
\pP_1(z)\pP_0(z)$ is paraunitary: $\pP_e(z)\pP_e^*(z)=I_{s}$;

\item[{\rm(2)}] $\pP$ can be represented as a product of  $\pP_0,\ldots,\pP_{J+1}$:
\begin{equation}\label{cascade}
\pP(z)=[I_r,\mathbf{0}] \pP_{J+1}(z) \pP_J(z)\cdots
\pP_1(z)\pP_0(z);
\end{equation}

\item[{\rm(3)}] $\pP_j, 1\le j\le J$ are elementary: $\pP_j(z)\pP_j^*(z)=I_s$ and $\cs(\pP_j)\subseteq[-1,1]$;

\item[{\rm(4)}] $(\pP_{j+1}, \pP_j)$ has mutually compatible symmetry for all $0\le j\le J$;

\item[{\rm(5)}] $\pP_{0}=\pU_{\sym \pth_2}^*$ and $\pP_{J+1}=\mbox{diag}(\pU_{\sym \pth_1}, I_{s-r})$, where $\pU_{\sym \pth_1}$, $\pU_{\sym \pth_2}$ are products of a permutation matrix with a diagonal matrix of monomials, as defined in
\eqref{sym:reorganize};

\item[{\rm(6)}] $J\le \displaystyle \max_{1\le m \le r, 1\le n\le s} \lceil {|\cs([\pP]_{m,n})|}/{2}\rceil$, where $\lceil \cdot \rceil$ is the ceiling function.
\end{enumerate}
\end{theorem}

The proof of Theorem \ref{thm:main:2} is given in
\cite{HanZhuangMat}, in which a general extension algorithm ($r\ge
1$) for the matrix extension with symmetry was also proposed. Since
we are only concerning about the construction of tight $\df$-wavelet
frame in this paper, we shall provide an simple algorithm for the
matrix extension with symmetry for the case $r=1$.

Before we introduce our algorithm, we need some auxiliary results.

For a $1\times n$ row vector $\vf$ in $\F$ such that $\|\vf\|\neq0$,
where $\|\vf\|^2:=\vf\vf^*$, we define $n_{\vf}$ to be the number of
nonzero entries in $\vf$ and $\e_j:=[0,\ldots,0,1,0,\ldots,0]$ to be
the $j$th unit coordinate row vector in $\R^n$. Let $E_{\vf}$ be a
permutation matrix such that $\vf
E_{\vf}=[f_1,\ldots,f_{n_{\vf}},0,\ldots,0]$ with $f_j\neq 0$ for
$j=1,\ldots, n_{\vf}$. We define
\begin{equation}
\label{Vf} V_{\vf}:=\left\{
                      \begin{array}{ll}
                        \frac{\bar{f_1}}{|f_1|}, & \hbox{if $n_{\vf}=1$;} \\
                        \frac{\bar{f_1}}{|f_1|}\left(I_n-\frac{2}{\|v_\vf\|^2}v_\vf^*v_{\vf}\right), & \hbox{if $n_{\vf}>1$,}
                      \end{array}
                    \right.
\end{equation}where
$v_\vf:=\vf-\frac{f_1}{|f_1|}\|\vf\|\e_1$. Observing that
$\|v_\vf\|^2=2\|\vf\|(\|\vf\|-|f_1|)$, we can verify that $V_\vf
V_{\vf}^*=I_n$ and $\vf E_{\vf} V_\vf =\|\vf\|\e_1$. Let
$U_{\vf}:=E_{\vf}V_{\vf}$. Then $U_\vf$ is unitary and satisfies
$U_{\vf}=[\frac{\vf^*}{\|\vf\|},F^*]$ for some $(n-1)\times n$
matrix $F$ in $\F$ such that $\vf U_{\vf}=[\|\vf\|,0,\ldots,0]$. We
also define $U_{\vf}:=I_n$ if $\vf={\bf0}$ and $U_{\vf}:=\emptyset$
if $\vf=\emptyset$. Here, $U_{\vf}$ plays  the role of reducing the
number of nonzero entries in $\vf$.

Let $\pth$ be a $1\times n$ row vector of Laurent polynomials with
symmetry such that $\sym\pth=[\gep_1 z^{c_1}, \ldots, \gep_n
z^{c_n}]$ for some $\gep_1, \ldots, \gep_n\in \{-1,1\}$ and $c_1,
\ldots, c_n\in \Z$.  Then, the symmetry of any entry in the vector
$\pth \mbox{diag}(z^{-\lceil c_1/2\rceil}, \ldots, z^{-\lceil
c_n/2\rceil})$ belongs to $\{ \pm 1, \pm z^{-1}\}$. Thus, there is a
permutation matrix $E_\pth$ to regroup these four types of
symmetries together so that
\begin{equation}\label{sym:reorganize}
\sym(\pth \pU_{\sym\pth})=[\mathbf{1}_{n_1}, -\mathbf{1}_{n_2},
z^{-1} \mathbf{1}_{n_3}, -z^{-1}\mathbf{1}_{n_4}],
\end{equation}
where $\pU_{\sym\pth}:= \mbox{diag}(z^{-\lceil c_1/2\rceil}, \ldots,
z^{-\lceil c_n/2\rceil})E_\pth$, $\mathbf{1}_m$ denotes the $1\times
m$ row vector $[1,\ldots,1]$,  and  $n_1, \ldots, n_4$ are
nonnegative integers uniquely determined by $\sym\pth$. Note that
$\pU_{\sym\pth}$ do not increase the length of the coefficient
support of $\pth$.

Let $\pq$ be a $1\times s$ row vector of Laurent polynomial
satisfying $\pq\pq^*=1$ and $\sym \pq=
[\mathbf{1}_{s_1},-\mathbf{1}_{ s_2}, z^{-1} \mathbf{1}_{
s_3},-z^{-1}\mathbf{1}_{ s_4}]$ for some nonnegative integers
$s_1,\ldots,s_4$ such that $s_1+s_2+s_3+s_4=s$. Then, $\pq$ must
take the form in  \eqref{eq:type1} or \eqref{eq:type2} with
$\vf_1\neq {\bf0}$ as follows:
\begin{equation}\label{eq:type1}
\begin{small}
\begin{aligned}
\pq &=[ \vf_1, -\vf_2,  \vg_1, -\vg_2]z^{\ell_1} +
[\vf_3,-\vf_4,\vg_3,-\vg_4 ]z^{\ell_1+1}
+\sum_{\ell=\ell_1+2}^{\ell_2-2}\coeff(\pq,\ell)z^\ell
\\&+ [\vf_3,\vf_4,\vg_1,\vg_2]z^{\ell_2-1} +
[\vf_1,\vf_2,\textbf{0},{\bf0}]z^{\ell_2};
\end{aligned}
\end{small}
\end{equation}
\begin{equation}\label{eq:type2}
\begin{small}
\begin{aligned}
 \pq
&=[{\bf0},{\bf0},\vf_1,- \vf_2]z^{\ell_1} + [ \vg_1,- \vg_2,\vf_3,-
\vf_4]z^{\ell_1+1} +\sum_{\ell=\ell_1+2}^{\ell_2-2}\coeff(\pq
,\ell)z^\ell \qquad\!\\&+ [ \vg_3, \vg_4, \vf_3,\vf_4]z^{\ell_2-1}+
[ \vg_1, \vg_2, \vf_1, \vf_2]z^{\ell_2}.
\end{aligned}
\end{small}
\end{equation}

If $\pq$ takes the form in \eqref{eq:type2}, we further construct a
permutation matrix $E_{\pq}$ so that $[ \vg_1, \vg_2, \vf_1,
\vf_2]E_{\pq}=[\vf_1, \vf_2,\vg_1,\vg_2]$ and  define
$\pU_{\pq,\gep}:= E_{\pq}\diag(I_{s-s_{\vg}},z^{-1} I_{s_{\vg}})$,
where $s_{\vg}$ is the size of the row vector $[\vg_1,\vg_2]$. Then
$\pq \pU_{\pq}$ takes the form in \eqref{eq:type1}.  For $\pq
E_{\gep}$ of form \eqref{eq:type1}, we simply let $\pU_{\pq}:=I_s$.
In this way, $\pq_0:=\pq\pU_{\pq}$ always takes the form in
\eqref{eq:type1} with $\vf_1\neq{\bf0}$.

Note that $\pU_{\pq}\pU_{\pq}^*=I_s$ and $\|\vf_1\|=\|\vf_2\|$ if
$\pq_0\pq_0^*=1$. Now  an $s\times s$ paraunitary matrix
$\pB_{\pq_0}$ to reduce the coefficient support of $\pq_0$  as in
\eqref{eq:type1} with $\vf_1\neq {\bf0}$ from $[\ell_1,\ell_2]$ to
$[\ell_1+1,\ell_2-1]$ is given by:

\begin{equation}\label{eq:MatrixForvectorDegBy2}
\pB_{\pq_0}^*:=\frac{1}{c}\left[
 \begin{array}{c|c|c|c}
 \vf_1(z+\frac{{c_0}}{c_{\vf_1}}+\frac1z)& \vf_2(z-\frac1z) & \vg_1(1+\frac1z) & \vg_2(1-\frac1z) \\
c F_1&  {\bf 0} &{\bf 0} &{\bf 0}\\
\hline
\noalign{\vspace{-0.1in}}&&&\\
-\vf_1(z-\frac1z) & -\vf_2(z-\frac{{c_0}}{c_{\vf_1}}+\frac1z) & -\vg_1(1-\frac1z) & -\vg_2(1+\frac1z) \\
{\bf 0}& c F_2  &{\bf 0} &{\bf 0}\\
\hline
\noalign{\vspace{-0.1in}}&&&\\
\frac{c_{\vg_1}}{c_{\vf_1}}\vf_1(1+z) & -\frac{c_{\vg_1}}{c_{\vf_1}}\vf_2(1-z) & c_{\vg_1'}\vg_1' & {\bf 0}\\
{\bf 0}&  {\bf 0} & c G_1&{\bf 0}\\
\hline
\noalign{\vspace{-0.1in}}&&&\\
\frac{c_{\vg_2}}{c_{\vf_1}}\vf_1(1-z) & -\frac{c_{\vg_2}}{c_{\vf_1}}\vf_2(1+z) & {\bf 0} & c_{\vg_2'} \vg_2'\\
{\bf 0}&  {\bf 0} &{\bf 0} &c G_2\\
\end{array}
\right],
\end{equation}
where $c_{\vf_1}:=\|\vf_1\|$,  $c_{\vg_1}:=\|\vg_1\|$, $
c_{\vg_2}:=\|\vg_2\|$, $c_0:=\coeff(\pq_0,\ell_1+1)\coeff(\pq_0^*,-\ell_2)/c_{\vf_1}$,
\begin{equation}
\begin{small}
\begin{aligned}
&c_{\vg_1'}:=
\begin{cases} \frac{-2c_{\vf_1}-\overline{c_0}}{c_{\vg_1}} &\text{if $\vg_1\neq{\bf0}$;}\\ c &\text{otherwise,}\end{cases}
\qquad
c_{\vg_2'}:=\begin{cases} \frac{2c_{\vf_1}-\overline{c_0}}{c_{\vg_2}} &\text{if $\vg_2\neq {\bf0}$;}\\
c &\text{otherwise,}
\end{cases}\\
&c:=(4c_{\vf_1}^2+2c_{\vg_1}^2+2c_{\vg_2}^2+|c_0|^2)^{1/2},
\end{aligned}
\end{small}
\end{equation}
and $[\frac{\vf_j^*}{\|\vf_j\|},F_j^*]=U_{\vf_j}$,
$[\vg_j'^*,G_j^*]=U_{\vg_j}$ are unitary constant extension matrices
in $\F$ for vectors $\vf_j,\vg_j$ in $\F$, for $j=1,2$,
respectively. The operations for the emptyset $\emptyset$ are
defined by $\|\emptyset\|=\emptyset$, $\emptyset+A=A$ and
$\emptyset\cdot A=\emptyset$ for any object $A$.

\begin{lemma}
\label{matrix:Bq} Let $\pq$ be a $1\times s$ row vector of Laurent
polynomial satisfying $\pq\pq^*=1$, $|\cs(\pq)|>2$, and $\sym
\pq=[\mathbf{1}_{s_1},-\mathbf{1}_{ s_2}, z^{-1} \mathbf{1}_{
s_3},-z^{-1}\mathbf{1}_{ s_4}]$ for some  nonnegative integers
$s_1,\ldots,s_4$ such that $s_1+s_2+s_3+s_4=s$. Let
$\pB_{\pq}:=\pU_{\pq}\pB_{\pq_0}\pU_{\pq}^*$ with $\pU_{\pq}$ and
$\pB_{\pq_0}$  being constructed as above. Then,
 $\sym\pB_{\pq}=[\mathbf{1}_{s_1},-\mathbf{1}_{ s_2},
z\mathbf{1}_{ s_3},-z\mathbf{1}_{
s_4}]^T[\mathbf{1}_{s_1},-\mathbf{1}_{ s_2}, z^{-1} \mathbf{1}_{
s_3},-z^{-1}\mathbf{1}_{ s_4}]$, $\cs({\pB}_\pq)=[-1,1]$, and
$\cs(\pq\pB_{\pq})=[\ell_1+1,\ell_2-1]$. That is, $\pB_\pq$ has
compatible symmetry with coefficient support on $[-1,1]$ and
$\pB_{\pq}$ reduces the length of the coefficient support of $\pq$
exactly by $2$. Moreover, $\sym(\pq\pB_{\pq})=\sym\pq$.
\end{lemma}

\begin{proof} Direct computation shows that
$\coeff(\pq\pB,k)=0$ for $k\in\{\ell_1-1,\ell_1,\ell_2+1,\ell_2\}$.
One can also show that   $\coeff(\pq\pB,k)\neq0$ for
$k\in\{\ell_1+1,\ell_2-1\}$. Hence
$\cs(\pq\pB_{\pq})=[\ell_1+1,\ell_2-1]$.  The other parts of the
Lemma  follows from our construction. \eop
\end{proof}

Now, it is easy to prove  Theorem \ref{thm:main:2} for $r=1$ using
Lemma \ref{matrix:Bq}. Moreover,  we can have a constructive
algorithm to derive $\pP_e$ from a given column $\pp$, which mainly
has three steps: initialization, support reduction, and
finalization. The step of initialization reduces the symmetry
pattern of $\pp$ to a standard form. The step of support reduction
is the main body of the algorithm, producing a sequence of
elementary matrices $\pA_1, \ldots, \pA_J$ that reduce the length of
the coefficient support of $\pp$ to $0$. The step of finalization
generates the desired matrix $\pP_e$ as in Theorem~\ref{thm:main:2}.
More precisely, our algorithm written in the form of
\emph{pseudo-code} is as follows:

\begin{algorithm} \label{alg:1} {
Input $\pp$, which is a $1\times s$ vector of Laurent polynomial
with symmetry satisfying $\pp\pp^*=1$.
\\
\noindent \textbf{\rm 1. Initialization:} Let  $\pq :=
\pp\pU_{\sym\pp}$. Then $\sym \pq=[{\bf1}_{ s_1},-{\bf1}_{ s_2},
z^{-1}{\bf1}_{ s_3},-z^{-1}{\bf1}_{ s_4}]$, where all nonnegative
integers $r_1,\ldots, r_4, s_1, \ldots, s_4$ are uniquely determined
by $\sym \pp$.
\\
\noindent \textbf{\rm 2. Support Reduction:} Let
$\pP_0:=\pU_{\sym\pp}^*$ and $J:=1$.
\begin{tabbing}
\hspace*{0.13in}\=\hspace{2ex}\=\hspace{2ex}\=\hspace{2ex}\=\hspace{2ex}\=\hspace{2ex}\=\hspace{2ex}\kill
\> {\rm \texttt{while}} $(|\cs(\pq)|>0)$ {\rm\texttt{do}}\\
\> \> {\rm\texttt{if}}  $|\cs(\pq)|>1$ {\rm\texttt{then}}\\
\>\>\>$\pq:=\pq\pB_{\pq}$, $\pA_J:=\pB_{\pq}$, and $\pP_J:=\pA_J^*$. \\
\> \> {\rm\texttt{else}}\\
\>\>\> $\pq$ is of the form: $ \pq =[ {\bf0}, {\bf0},  \vg_1,
-\vg_2]z^{-1} + [\vf_3,{\bf0},\vg_1,
\vg_2]; $ Let\\
\>\>\> \qquad$ \pB_J:= \left[
   \begin{array}{c|cc|cc}
     I_{s_1+s_2}&  &  &  &  \\
     \hline
     & \frac{\vg_1^*(1+z)}{2\|\vg_1\|} & G_1 & \frac{\vg_1^*(1-z)}{2\|\vg_1\|} & {\bf0} \\
     \hline
     & \frac{\vg_2^*(1-z)}{2\|\vg_1\|} & {\bf0} & \frac{\vg_2^*(1+z)}{2\|\vg_1\|} & G_2\\
   \end{array}
 \right]$,\\
 \>\>\>
 where $[\frac{\vg_j^*}{\|\vg_j\|}, G_j] = U_{\vg_j},
 j=1,2$.  $\pB_J$ is paraunitary, $|\cs(\pq\pB_J)|=0$, and\\
\>\>\>\qquad\qquad$\sym(\pq\pB_J)=[{\bf1}_{ s_1},-{\bf1}_{ s_2},
1,z^{-1}{\bf1}_{ s_3-1},-1,-z^{-1}{\bf1}_{ s_4-1}]$.\\
\>\>\> Let $E$ be a permutation matrix such that\\
 \>\>\>\qquad\qquad
$\sym(\pq\pB_J)E=[{\bf1}_{ s_1+1},-{\bf1}_{ s_2+1}, z^{-1}{\bf1}_{
s_3-1},-z^{-1}{\bf1}_{ s_4-1}]$.\\
\>\>\> Let $\pA_J:=\pB_J E$, $\pP_J:=\pA_J^*$, and $\pq=\pq\pA_J$.
\\
\> \> {\rm\texttt{end if}}\\
\>\> $ J := J+1$\\
\> {\rm\texttt{end while}}\\
\end{tabbing}
\noindent \textbf{\rm 3. Finalization:} $\pq =
[\vf,{\bf0},{\bf0},{\bf0}]$ for some $1\times s_1'$  constant
vectors $\vf$ in $\F$. Let $U:=\diag(U_{\vf},I_{s-s_1'})$. Define
$\pP_{J}:=U_{\vf}^*$.
\\
\noindent Output a desired matrix $\pP_e$
satisfying all the properties in Theorem~\ref{thm:main:2} for $r=1$.
}\end{algorithm}

\section{Symmetric Complex Tight $\df$-frame via Matrix Extension}

In this section, we shall discuss the application of our results on
matrix extension with symmetry to $\df$-band symmetric paraunitary
filter banks in electronic engineering and to orthonormal
multiwavelets with symmetry in wavelet analysis. In order to do so,
let us introduce some definitions first.

Let $\F$ be a subfield of $\C$ such that \eqref{F} holds. Let $\phi$
be a compactly supported $\df$-refinable function  in $L_2(\R)$
associated with a low-pass filter $a_0:\Z\mapsto\F$.  Recall that
its symbol is defined to be $\pa_0(z):=\sum_{k\in \Z} \ta_0(k) z^k$,
which is a Laurent polynomials with coefficients in $\F$. We denote
its \emph{$\df$-band subsymbols} by $\pa_{0; \gamma}(z):=\sqrt{\df}
\sum_{k\in\Z} \ta_0(\gamma+\df k) z^k$, $\gamma=0,\ldots,\df-1$. To
construct a tight $\df$-wavelet frame from $\phi$,  one has to
design high-pass filters $\ta^1, \ldots, \ta^L: \Z \mapsto
\F^{\mphi\times \mphi}$ such that \eqref{UEP} holds. Let $\PP(z)$ be
the polyphase matrix defined by
\begin{equation}\label{polyphase}
\PP(z):=\left[ \begin{matrix} \pa_{0;0}(z) &\cdots &\pa_{0; \df-1}(z)\\
\pa_{1;0}(z) &\cdots &\pa_{1; \df-1}(z)\\
\vdots &\vdots &\vdots\\
\pa_{L;0}(z) &\cdots &\pa_{L; \df-1}(z)
\end{matrix}\right].
\end{equation}
Then it is easily seen that \eqref{UEP} is equivalent to
\begin{equation}\label{UEPpolyphase}
\PP(z)^* \PP(z)=I_{\df},
\end{equation}
where each $\pa_{m;\gamma}$ is a subsymbol of $a^m$ for
$m=1,\ldots,L; \gamma=0,\ldots,\df-1$, respectively.

Symmetry of the filters is a very much desirable property in many
applications. We say that the low-pass filter $\pa_0$ (or $\ta_0$)
has symmetry if
\begin{equation}\label{mask:sym}
\pa_0(z)= z^{(\df-1)c_0} \pa_0(1/z)
\end{equation}
for some  $c_0\in \R$ such that $(\df-1)c_0\in\Z$. If $\pa_0$ has
symmetry as in \eqref{mask:sym} and if $1$ is a simple eigenvalue of
$\pa_0(1)$, then it is well known that the $\df$-refinable function
$\phi$ associated with the low-pass filter $\pa_0$ has symmetry:
$\phi(c_0-\cdot)= \phi$.

Under the symmetry condition in \eqref{mask:sym},
\begin{equation}\label{eq:polyPhaseSymCondition}
\pa_{0;\gamma}(z)=z^{R_\gamma}\pa_{0;{Q_\gamma}}(z^{-1})\,\quad
\gamma=0,\ldots,\df-1,
\end{equation}
where $\gamma,Q_\gamma\in \Gamma:=\{0,\ldots,\df-1\}$ and
$R_\gamma$, $Q_\gamma$ are uniquely determined by
\begin{equation}\label{RQ}
(\df-1) c_0-\gamma=\df R_\gamma+Q_\gamma \quad \mbox{with} \quad
R_\gamma\in \Z, \; Q_\gamma\in\Gamma.
\end{equation}
Now, we can easily deduce symmetric Laurent polynomial via
\eqref{eq:polyPhaseSymCondition}. In fact, let
\begin{equation}\label{eq:symmetrization}
\pb_{0;\gamma}(z):=
\begin{cases}
      \pa_{0;\gamma}(z), & \hbox{$\gamma=Q_\gamma;$} \\
       \frac{1}{\sqrt{2}}(\pa_{0;\gamma}(z)+ z^{\kappa_{\gamma}}\pa_{0;{Q_\gamma}}(z)), & \hbox{$\gamma<Q_\gamma$;}\\
       \frac{1}{\sqrt{2}}(\pa_{0;\gamma}(z)- z^{\kappa_{\gamma}}\pa_{0;{Q_\gamma}}(z)), & \hbox{$\gamma>Q_\gamma$,}\\
\end{cases}
\end{equation}
where $\gamma=0,\ldots,\df-1$ and each $\kappa_\gamma$ is an integer
so that the length of the coefficient support of $\pa_{0;\gamma}(z)+
z^{\kappa_{\gamma}}\pa_{0;{Q_\gamma}}(z)$ is minimal. Then, one can
show that each $\pb_{0;\gamma}(z)$ is of a Laurent polynomial with
symmetry. Note that the transform matrix $\pU$ with respect to
\eqref{eq:symmetrization} is paraunitary.

Next, we show that we can construct a vector of Laurent polynomial
with symmetry from a low-pass filter $a$ for the complex pseudo
spline of type I with order $(m,n)$ to which Algorithm \ref{alg:1}
is applicable.

Fixed $m,n\in\N$ such that $1\le 2n-1\le m$. For simplicity, let
$a_0=a$ and $\pa_0$ be its symbol. From \eqref{eq:symmetrization},
we can construct a $1\times \df$ vector of Laurent polynomial
$\pp(z)=[\pb_{0;0}(z),\ldots,\pb_{0;\df-1}(z)]$ from $\pa_0$.
However, $\pp\pp^*\neq 1$ when $n<m$. To apply our matrix extension
algorithm, we need to append extra entries to $\pp$.
 It is easy to show that
\[
1-\sum_{j=0}^{\df-1}|\wh a_0(\xi+2j\pi/\df)|^2 =
1-\sum_{\gamma=0}^{\df-1}\pa_{0;\gamma}(z^\df)\pa_{0;\gamma}^*(z^\df),
\quad z=e^{-\iu\xi},
\]
where $\pa_{0;\gamma}, \gamma=0,\ldots,\df-1$ are the subsymbols of
$a_0$.  By Corollary \ref{cor:maskRemainder}, we have
\[
1-\sum_{j=0}^{\df-1}|\wh a_0(\xi+2j\pi/\df)|^2=|\wh b_0(\df\xi)|^2.
\]
for some $2\pi$-periodic trigonometric function $\wh b_0$ with real
coefficients. Hence, we can construct a Laurent polynomial
$\pa_{0;\df}(z)$ from $\wh b_0$ such that $\pa_{0;\df}(e^{-\iu\xi})
= \wh b_0(\xi)$. Then, the vector of Laurent polynomials
$\pq(z)=[\pa_{0;0}(z),\ldots,\pa_{0;\df-1}(z),\pa_{0;\df}(z)]$
satisfies $\pq\pq^*=1$.

For $m=2n$, by Corollary \ref{cor:maskRemainder}, $| {\wh
b_0}(\xi)|^2 = c_{2n,2n-1}[\sin^2(\xi/2)/\df^2]^{2n-1}$. W
$\pa_{0;\df}(z)$ can be constructed explicitly  as follows:
\[
\pa_{0;\df}(z) =
\sqrt{c_{2n,2n-1}}\left(\frac{2-z-1/z}{4\df^2}\right)^{n-1}\frac{1-z}{2\df}.
\]
Then $\pa_{0;\df}(z)$ is symmetric, let
$\pb_{0;\df}(z)=\pa_{0;\df}(z)$. Then
$\pp:=[\pb_{0;0},\ldots,\pb_{0;\df}]$ is a $1\times (\df+1)$ vector
of Laurent polynomials with symmetry satisfying $\pp\pp^*=1$.

 If $m\neq 2n$,
$\pa_{0;\df}(z)$ is a not necessary symmetric, we can further let
$\pb_{0;\df}(z) =(\pa_{0;\df}(z)+\pa_{0;\df}(1/z))/2$ and
$\pb_{0;\df+1}(z) =(\pa_{0;\df}(z)-\pa_{0;\df}(1/z))/2$. In this
way, $\pp:=[\pb_{0;0},\ldots,\pb_{0;\df},\pb_{0;\df+1}]$ is a
$1\times (\df+2)$ vector of Laurent polynomials with symmetry
satisfying $\pp\pp^*=1$.

Consequently, we can summerize the above result as follows:
\begin{theorem}\label{thm:symPa}
Let $m,n\in\N$ be such that $1\le n\le m$. Let $a_0$ be the low-pass
filter for the complex pseudo spline of order $(m,n)$ defined in
\eqref{def:Amn} and $\pa_0(z)$ be its symbol. Then, one can derive
Laurent polynomials $\pa_{0;\df}(z),\ldots,\pa_{0;L}(z), \df\le L\le
\df+1$ such that
$\pp_{\pa_0}:=[\pa_{0;0},\ldots,\pa_{0;\df-1},\cdots,\pa_{0;L}(z)]$
satisfies $\pp_{\pa_0}\pp_{\pa_0}^*=1$, where $\pa_{0;0}, \ldots,
\pa_{0; \df-1}$ are subsymbols of  $a_0$. Moreover, one can
construct a  $(L+1) \times (L+1) $ paraunitary matrix $\pU$ such
that $\pp_{\pa_0}\pU$ is  a vector of Laurent polynomial with
symmetry.
In particular, for $m=2n$,
$\pa_{0;\df}(z)=\sqrt{c_{2n,2n-1}}\left(\frac{2-z-1/z}{4\df^2}\right)^{n-1}\frac{1-z}{2}$
and $L=\df$.
\end{theorem}

Now,  applying Theorem \ref{thm:main:2}, we have the following
algorithm to construct high-pass filters $\pa_1,\ldots,\pa_{L}$ from
a low-pass filter $a_0$ for a complex pseudo spline of type I with
order $(m,n)$.

\begin{algorithm}
\label{alg:2} Input low-pass  filter $\pa_0$ for a complex pseudo
spline of type I with order $(m,n)$, $1\le 2n-1\le m$. $\pa_0$
satisfies \eqref{mask:sym}.
\begin{itemize}
\item[{\rm(1)}] Construct  $\pp_{\pa_0}(z)$ and $\pU$ as in Theorem \ref{thm:symPa} such
that $\pp:=\pp_{\pa_0}\pU$ is a $1\times (L+1)$ row vector of
Laurent polynomials with symmetry
 ($L=\df$ when $m=2n$ while
$L=\df+1$ when $m\neq 2n$.

\item[{\rm(2)}] Derive $\pP_e$ with all the properties as in Theorem~\ref{thm:main:2} for the case $r=1$ from $\pp$ by
Algorithm~\ref{alg:1}.

\item[{\rm(3)}] Let $\PP:=\pP_e\pU^*=:(\pa_{m;\gamma})_{0\le m\le L,0\le\gamma\le \df-1}$ as  in \eqref{polyphase}. Define high-pass filters
\begin{equation}
\label{highpass:def} \pa_m(z):=\frac{1}{\sqrt{\df}}
\sum_{\gamma=0}^{\df-1}\pa_{m;\gamma}(z^\df)z^\gamma, \qquad m=1,
\ldots, L.
\end{equation}
Note that we only need the first $\df$ columns of $\pP_e\pU^*$.
\end{itemize}
Output a symmetric filter bank $\{ \pa_0, \pa_1, \ldots, \pa_{L}\}$
with the perfect reconstruction property, i.e.
$\PP^*(z)\PP(z)=I_\df$. All filters $\pa_m$, $m=1, \ldots, \df-1$,
have symmetry:
\begin{equation}\label{highpass:sym}
\pa_{m}(z)=\gep_m z^{\df c_m-c_0}\pa_m(1/z),
\end{equation}
where $c^m:=k_m+c_0\in\R$ and all $\varepsilon_m\in\{-1,1\}$,
$k_m\in\Z$ for  $m=1,\ldots,L$ are determined by the symmetry
pattern of $\pP_e$ as follows:
\begin{equation}
\label{sym:Pe} [1,\gep_1z^{k_1},\ldots,\gep_{L}
z^{k_L}]^T\sym\pp:=\sym\pP_e.
\end{equation}

\end{algorithm}

\begin{proof} Rewrite $\pP_e=(\pb_{m;\gamma})_{0\le m,\gamma \le L}$. Since $\pP_e$ has
compatible symmetry as in \eqref{sym:Pe}, we have $\sym
\pb_{m;\gamma}=\varepsilon_m z^k_{m}\sym \pb_{0;\gamma}$ for
$m=1,\ldots,\df-1$. By \eqref{eq:symmetrization}, we have $\sym
\pb_{0;\gamma} = \sgn(Q_\gamma-\gamma)z^{R_\gamma+k_\gamma}$,
$\gamma = 0,\ldots,\df-1$, where $\sgn(x)=1$ for $x\ge0$ and
$\sgn(x)=-1$ for $x<0$. Then, we have
\begin{equation}
\label{eq:symCompatibleB}
\sym\pb_{m;\gamma}=\sgn(Q_\gamma-\gamma)\varepsilon_m
z^{R_\gamma+\kappa_{\gamma}+k_m}, \
\end{equation}
By \eqref{eq:symCompatibleB} and the transformation matrix $\pU^*$
with respect to \eqref{eq:symmetrization}, we deduce that
\begin{equation}\label{eq:compatibleCondition}
\pa_{m;\gamma} =
\varepsilon_mz^{R_\gamma+k_m}\pa_{m;{Q_\gamma}}(z^{-1}).
\end{equation}
This implies that $\sym\pa_m=\varepsilon_m z^{\df(k_m+c_0)-c_0}$,
which is equivalent to \eqref{highpass:sym} with $c_m:=k_m+c_0$ for
$m=1,\ldots,L$.
\end{proof}

Since the high-pass filters $\pa_1,\ldots,\pa_L$ satisfy
\eqref{highpass:sym}, it is easy to verify that
$\psi^1,\ldots,\psi^L$ defined in \eqref{waveletGenerators} also has
the following symmetry:
\begin{equation}\label{sym:psi}
\psi^1(c_1-\cdot)=\varepsilon_1\psi^{1},\quad\psi^2(c_2-\cdot)=\varepsilon_2\psi^{2},\quad
\ldots, \quad\psi^L(c_L-\cdot)=\varepsilon_L\psi^{1},.
\end{equation}
In fact, by \eqref{highpass:sym}, \eqref{waveletGenerators} and the
symmetry of $\phi$, we have
\[
\begin{aligned}
\wh{\psi^m}(-\df \xi) &= \pa_m(e^{\iu\xi})\wh{\phi}(-\xi)=
\gep_m e^{-\iu(\df c_m-c_0)}\pa_m(e^{-\iu\xi})e^{-\iu c_0\xi}\wh{\phi}(\xi)\\
&=\gep_m e^{-\iu (\df
c_m)\xi}\pa_m(e^{-\iu\xi})\wh{\phi}(\xi)=\gep_m e^{-\iu (\df
c_m)\xi}\wh{\psi^m}(\df \xi).
\end{aligned}
\]
This implies $\wh{\psi^m}(-\xi) =\gep_m e^{-\iu
c_m\xi}\wh{\psi^m}(\xi)$, which is equivalent to \eqref{sym:psi}.

In the following, let us present some examples to demonstrate our
results and illustrate our algorithms.
\begin{example}{\rm
\label{ex:1} Consider dilation factor $\df=2$. Let  $m=4, n=2$. Then
$P_{4,3}(y) = 1+4y+4y^2$. The low-pass filter $a_0^{4,2}$  with its
symbol $\pa_0$  for the complex pseudo spline of order $(4,2)$ is
given by
\[
\pa_0(z)
=\left(\frac{1+z}{2}\right)^4\left[\left(-\frac12-\frac{\sqrt{6}}{4}\iu\right)\frac{1}{z}+\left(2+\frac{\sqrt{6}}{2}\iu\right)\frac{1}{z^2}+\left(-\frac12-\frac{\sqrt{6}}{4}\iu\right)\frac{1}{z^3}\right].
\]
Note that $\cs(\pa_0)=[-3,3]$ and $\pa(z)=\pa(z^{-1})$. In this
case, $m=2n$, by Theorem~\ref{thm:symPa}, we can obtain
$\pp_{\pa_0}=[\pa_{0;0}(z),\pa_{0;1}(z),\pa_{0;2}(z)]$ as follows:
\[
\begin{aligned}
\pa_{0;0}(z) &= {-\frac{\sqrt {3}\iu}{48} \left( 3{z}+8\sqrt {6}\iu
-6 +3
z^{-1}\right)};\\
\pa_{0;1}(z) &= -\frac{(\sqrt{2}+\sqrt{3}\iu)}{160}(5z+8\sqrt{6}\iu-26+5z^{-1})(1+z^{-1});\\
\pa_{0;2}(z) &= -\frac{\sqrt{5}}{32}(z-2+z^{-1})(1-z^{-1}).
\end{aligned}
\]
$\pp_{\pa_0}$ is already a $1\times 3$ vector of Laurent polynomials
with symmetry, i.e., $\sym\pp_{\pa_0}=[1,z^{-1},-z^{-1}]$. Applying
Algorithm~\ref{alg:2}, we can obtain its $3\times 3$ extension
matrix $\pP_e^*=[\pp_{\pa_0}^*,\pp_{\pa_1}^*,\pp_{\pa_2}^*]$ with
$\pp_{\pa_1}:=[\pa_{1;0},\pa_{1;1},\pa_{1;2}]$ and
$\pp_{\pa_2}:=[\pa_{2;0},\pa_{2;1},\pa_{2;2}]$ as follows:
\[
\begin{aligned}
\pa_{1;0}(z) &= {\frac{-\sqrt {35}(\sqrt{6}+3\iu)}{1680} \left(
3{z}-8\sqrt {6}\iu +42 +3
z^{-1}\right)};\\
\pa_{1;1}(z) &= {\frac{-\sqrt {35}(2\sqrt{6}+\iu)}{5600} \left(
5{z}-8\sqrt {6}\iu -74 +5
z^{-1}\right)(1+z^{-1})};\\
\pa_{1;2}(z) &= {\frac{-\sqrt {7}(\sqrt{3}-\sqrt{2}\iu)}{1120}
\left( 5{z}-16\sqrt {6}\iu -58 +5 z^{-1}\right)(1-z^{-1})};\\
\pa_{2;0}(z) &= \frac{\sqrt{42}\iu}{56}(z-z^{-1});\\
\pa_{2;1}(z) &= {\frac{\sqrt {7}(2+\sqrt{6}\iu)}{560} \left(
5{z}-4\sqrt {6}\iu -22 +5
z^{-1}\right)(1-z^{-1})};\\
\pa_{2;2}(z) &={\frac{\sqrt {70}(2\sqrt{6}+\iu)}{560} \left(
5{z}-12\sqrt {6}\iu -26 +5 z^{-1}\right)(1+z^{-1})}.
\end{aligned}
\]
Note that the symmetry of $\pP_e$ satisfies
$\sym\pP_e=[1,1,-1]^T\sym\pp_{\pa_0}$ and the coefficient support of
$\pP_e$ satisfies $\cs([\pP_e]_{:,j})\subseteq\cs([\pp_{\pa_0}]_j)$
for $j=1,2,3$.
 The high-pass filters $\pa_1,\pa_2$ constructed from $\pp_{\pa_1}$
and $\pp_{\pa_2}$ via \eqref{highpass:def} are then given by:
\[
\begin{aligned}
 \pa_1(z)
 &=-\frac{\sqrt{35}(4\sqrt{3}+\sqrt{2}\iu)}{11200}(5z+26+2\sqrt{6}\iu+5z^{-1})(z-2+z^{-1})^2;\\
 \pa_2(z)&
 =\frac{\sqrt{7}(\sqrt{2}+\sqrt{3}\iu)}{560}(5z+16+2\sqrt{6}\iu+5z^{-1})(z-2+z^{-1})(z-z^{-1}).
\end{aligned}
\]
We have $\pa_1(z)=\pa_1(z^{-1})$ and $\pa_2(z)=-\pa_2(z^{-1})$. Let
$\phi$ be the $2$-refinable function associated with the low-pass
filter $\pa_0$. Let $\psi^1,\psi^2$ be the wavelet functions
associated with the high-pass filters $\pa_1,\pa_2$ respectively by
\eqref{waveletGenerators}. Then $\phi(-\cdot)=\phi$,
$\psi^1(-\cdot)=\psi^1$ and $\psi^2(-\cdot)=-\psi^2$. The graphs of
$\phi,\psi^1,\psi^2$ are as follows:
\begin{figure}[th]
\centerline{\scalebox{1.0}{
\hbox{\epsfig{file=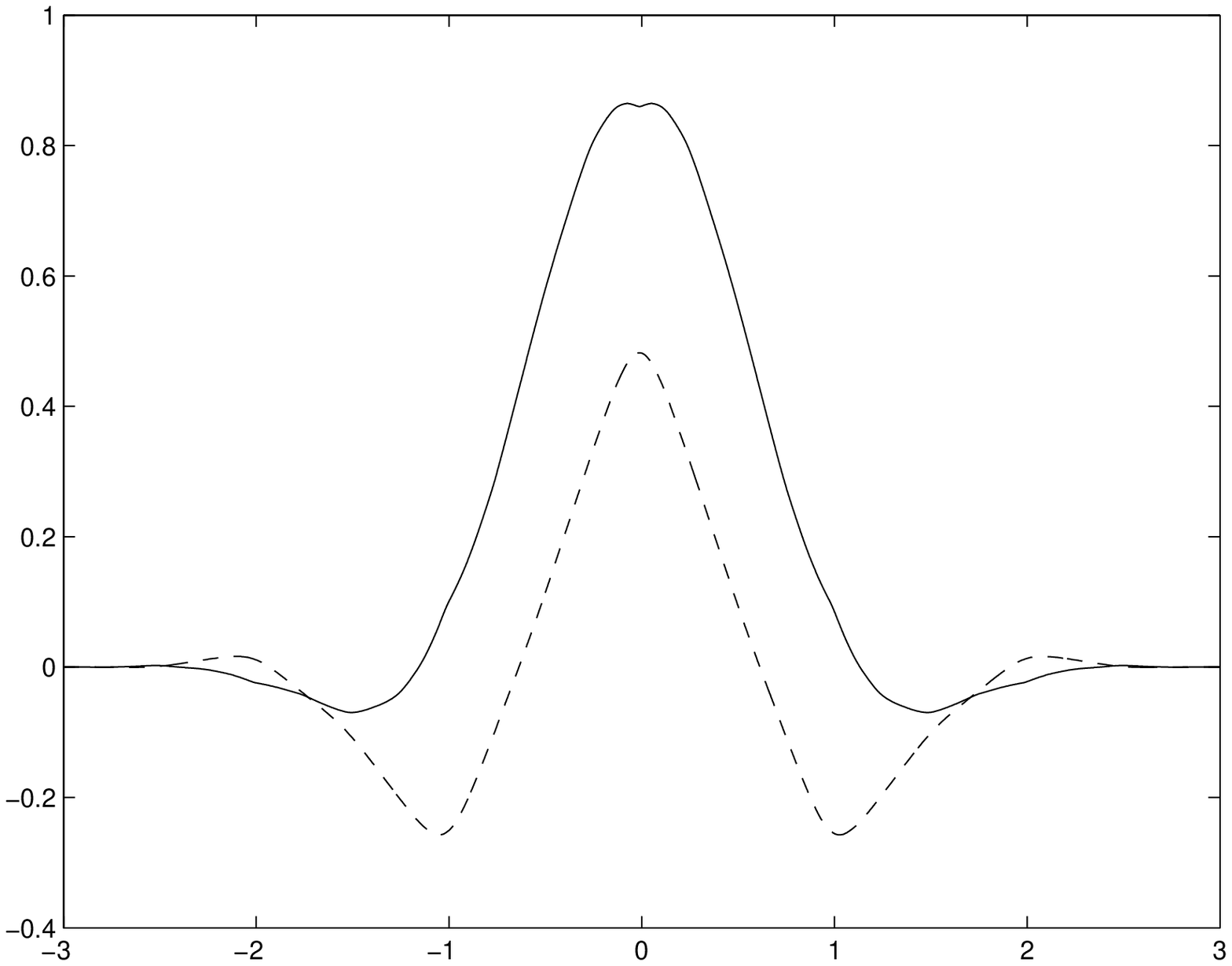,width=1.2in,height=1.0in}
\epsfig{file=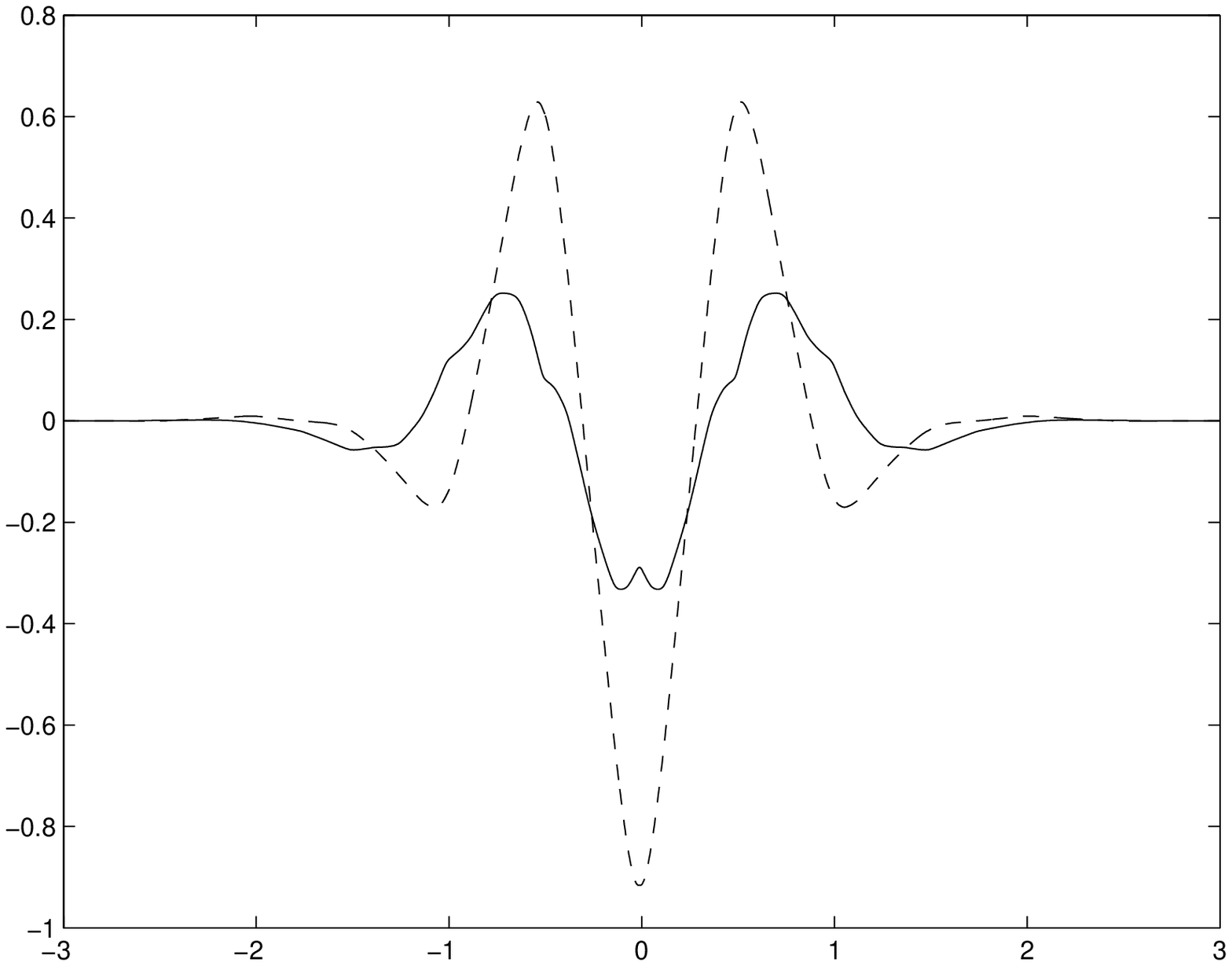,width=1.2in,height=1.0in}
\epsfig{file=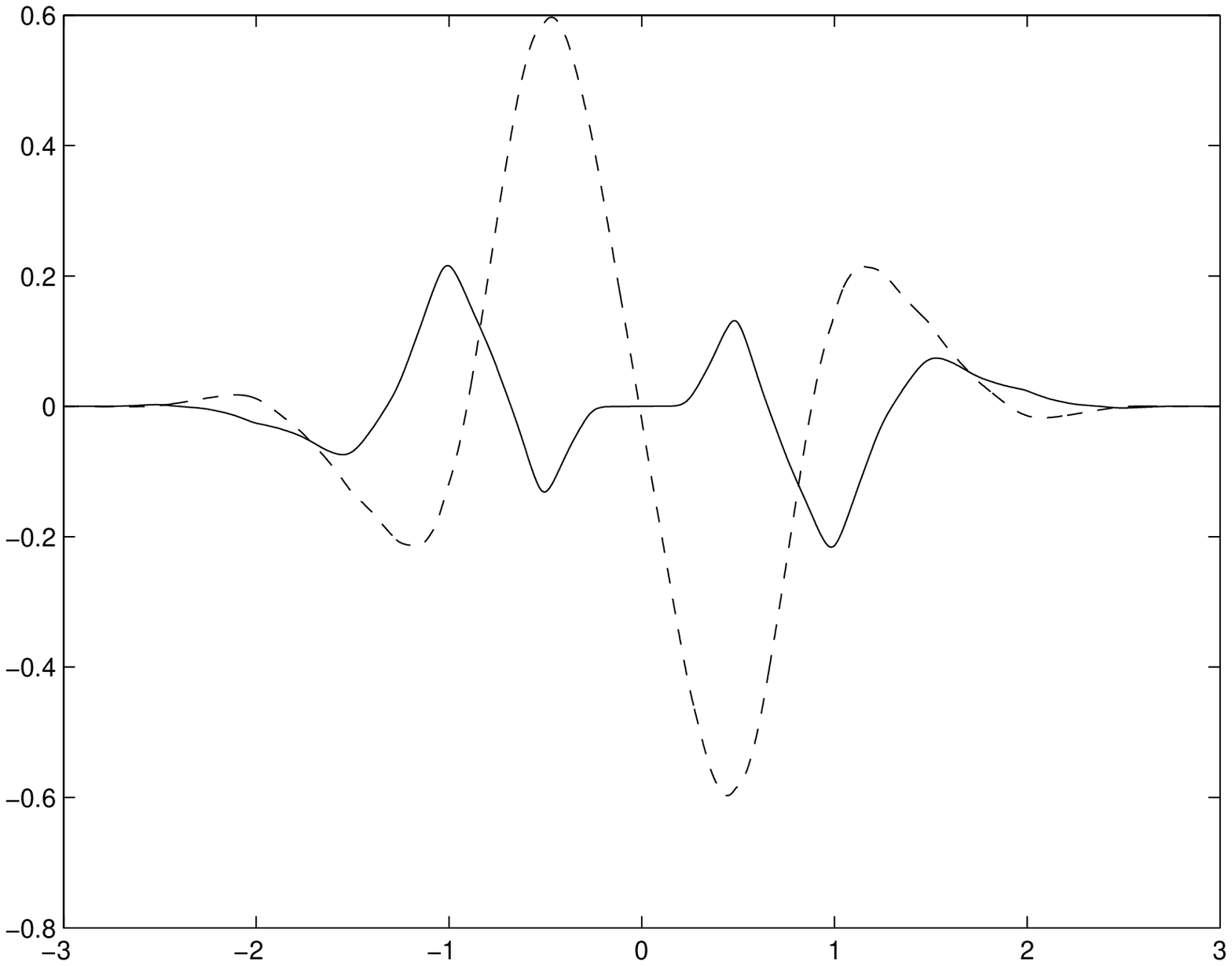,width=1.2in,height=1.0in} }}}
\begin{caption}
{ The graphs of $\phi,\psi^1,\psi^2$ (left to right). Real
part:solid line. Imaginary part:dashed line.}
\end{caption}
\end{figure}
}
\end{example}

\begin{example}{\rm
\label{ex:2} Consider dilation factor $\df=3$. Let  $m=4, n=2$. Then
$P_{4,3}(y) =1+\frac{32}{3}y+64y^2)$. The low-pass filter
$a_0^{4,2}$ with its symbol $\pa_0$  for the complex pseudo spline
of order $(4,2)$ is given by
\[
\pa_0(z)
=\left(\frac{\frac1z+1+z}{3}\right)^4\left[\left(-\frac43-\frac{2\sqrt{5}}{3}\iu\right)\frac{1}{z}+\left(\frac{11}{3}+\frac{4\sqrt{5}}{3}\iu\right)+\left(-\frac43-\frac{2\sqrt{5}}{3}\iu\right)z\right].
\]
Note that $\cs(\pa_0)=[-5,5]$ and $\pa(z)=\pa(z^{-1})$. In this
case, $m=2n$, by Theorem~\ref{thm:symPa}, we can obtain
$\pp_{\pa_0}=[\pa_{0;0}(z),\pa_{0;1}(z),\pa_{0;2}(z),\pa_{0;3}(z)]$
as follows:
\[
\begin{aligned}
\pa_{0;0}(z) &= {-\frac{\sqrt {15}\iu}{405} \left( 10{z}+27\sqrt
{5}\iu -20 +10
z^{-1}\right)};\\
\pa_{0;1}(z) &= \frac{\sqrt{3}}{243}(-(4+2\sqrt{5}\iu)  z^{-2}+30z^{-1}+60+6\sqrt{5}\iu-(5+4\sqrt{5}\iu)z);\\
\pa_{0;2}(z) &=
\frac{\sqrt{3}}{243}(-(5+4\sqrt{5}\iu)z^{-2}+(60+6\sqrt{5}\iu)z^{-1}+30-(4+2\sqrt{5}\iu)
z)\\
\pa_{0;3}(z) & = -\frac{2\sqrt{10}}{81}(z-2+z^{-1})(1-z).
\end{aligned}
\]
Note that $\pa_{0;1}(z) = z^{-1}\pa_{0;2}(z^{-1})$. Let $\pp =
\pp_{\pa_0} \pU$ with $\pU$ being the paraunitary matrix given by
\[
\pU:=\left[
       \begin{array}{cccc}
         1 & 0 & 0 & 0 \\
         0 & \frac{1}{\sqrt{2}} & \frac{1}{\sqrt{2}} & 0 \\
         0 & \frac{1}{\sqrt{2}} & -\frac{1}{\sqrt{2}} & 0 \\
         0 & 0 & 0 & z^{-1} \\
       \end{array}
     \right].
\]
Then $\pp$ is  a $1\times 4$ vector of Laurent polynomials with
symmetry pattern satisfying $\sym\pp=[1,z^{-1},-z^{-1},-z^{-1}]$.
Applying Algorithm~\ref{alg:2}, we can obtain a $4\times 4$
extension matrix
$\pP_e^*=[\pp_{\pa_0}^*,\pp_{\pa_1}^*,\pp_{\pa_2}^*,\pp_{\pa_3}]$
with $\pp_{\pa_1}:=[\pa_{1;0},\pa_{1;1},\pa_{1;2},\pa_{1;3}]$,
$\pp_{\pa_2}:=[\pa_{2;0},\pa_{2;1},\pa_{2;2},\pa_{2;3}]$, and
$\pp_{\pa_3}:=[\pa_{3;0},\pa_{3;1},\pa_{3;2},\pa_{3;3}]$.
The coefficient support of $\pP_e$ satisfies
$\cs([\pP_e]_{:,j})\subseteq\cs([\pp_{\pa_0}]_j)$ for $j=1,2,3,4$.
 The high-pass filters $\pa_1,\pa_2,\pa_3$ constructed from $\pp_{\pa_1}$
, $\pp_{\pa_2}$, and $\pp_{\pa_3}$ via \eqref{highpass:def} are then
given by:
\[
\begin{aligned}
 \pa_1(z)
 &=\frac{\sqrt{19178}}{4660254}(b_1(z)+b_1(z^{-1}));
\\
 \pa_2(z)& =\frac{\sqrt{218094}}{17665614}(b_2(z)-b_2(z^{-1}));
\\
 \pa_3(z)& =\frac{2\sqrt{1338}}{54189}(b_3(z)-z^3b_3(z^{-1})).
\end{aligned}
\]
where
\[
\begin{aligned}
b_1(z)&=\left( -172-86\,i\sqrt {5} \right) {z}^{5}+ \left(
-215-172\,i\sqrt {5} \right) {z}^{4}-258\,i\sqrt {5}{z}^{3}
\\& +\left( 1470+1224\,i\sqrt {5 } \right) {z}^{2}+ \left(
1860+2328\,i\sqrt {5} \right) z-3036\,i \sqrt {5}-2943\\
b_2(z)& = \left( -652-326\,i\sqrt {5} \right) {z}^{5}+ \left(
-815-652\,i\sqrt {5} \right) {z}^{4}-978\,i\sqrt {5}{z}^{3}
\\&+
\left( 1832\,i\sqrt {5}+ 1750 \right) {z}^{2}+ \left( 3508\,i\sqrt
{5}+3020 \right) z\\
b_3(z)& = \left( 4\,\sqrt {5}+10\,i \right) {z}^{5}+ \left( 5\,\sqrt
{5}+20\,i \right) {z}^{4}+30\,i{z}^{3}+ \left( -53\,\sqrt {5}-260\,i
\right) {z }^{2}.
\end{aligned}
\]
We have $\pa_1(z)=\pa_1(z^{-1})$, $\pa_2(z)=-\pa_2(z^{-1})$, and
$\pa_3(z)=-z^3\pa_3(z^{-1})$. Let $\phi$ be the $3$-refinable
function associated with the low-pass filter $\pa_0$. Let
$\psi^1,\psi^2,\psi^3$ be the wavelet functions associated with the
high-pass filters $\pa_1,\pa_2,\pa_3$ respectively by
\eqref{waveletGenerators}. Then $\phi(-\cdot)=\phi$,
$\psi^1(-\cdot)=\psi^1$, $\psi^2(-\cdot)=-\psi^2$, and
$\psi^3(1-\cdot)=-\psi^3$ . See Figure~2 for the graphs of
$\phi,\psi^1,\psi^2$, and $\psi^3$.
\begin{figure}[th]
\centerline{\scalebox{1.0}{ \hbox{
\epsfig{file=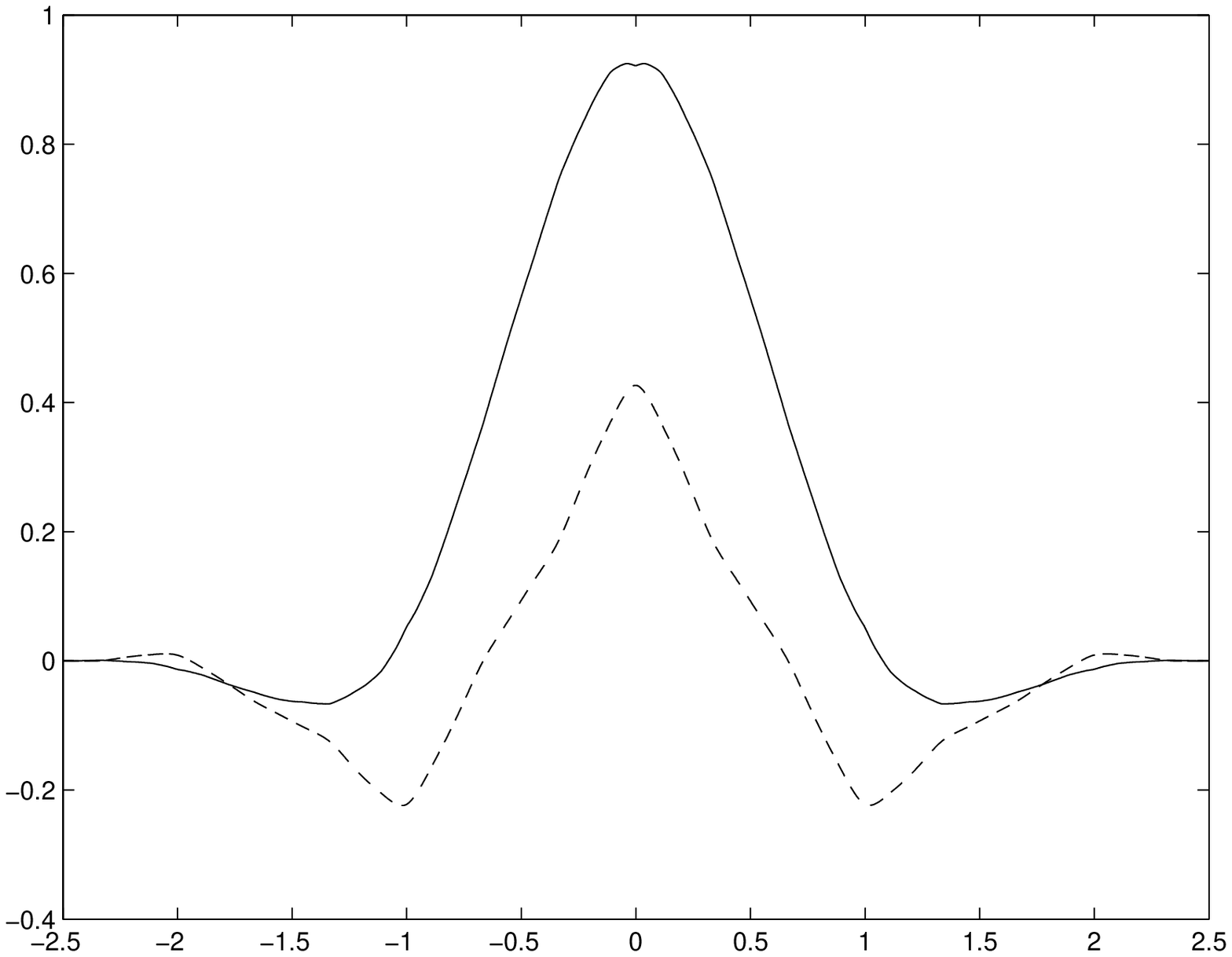,width=1.1in,height=1.0in}
\epsfig{file=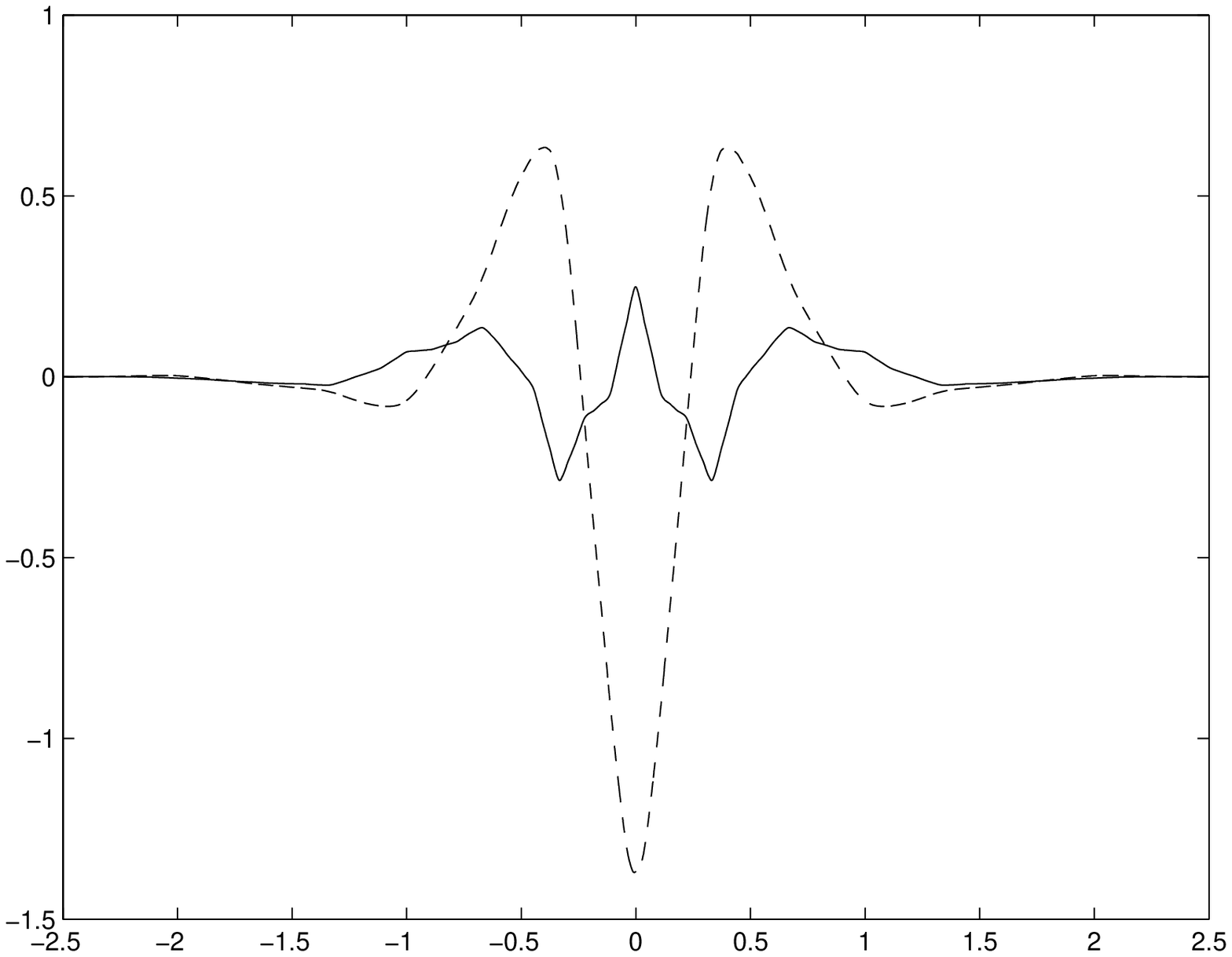,width=1.1in,height=1.0in}
\epsfig{file=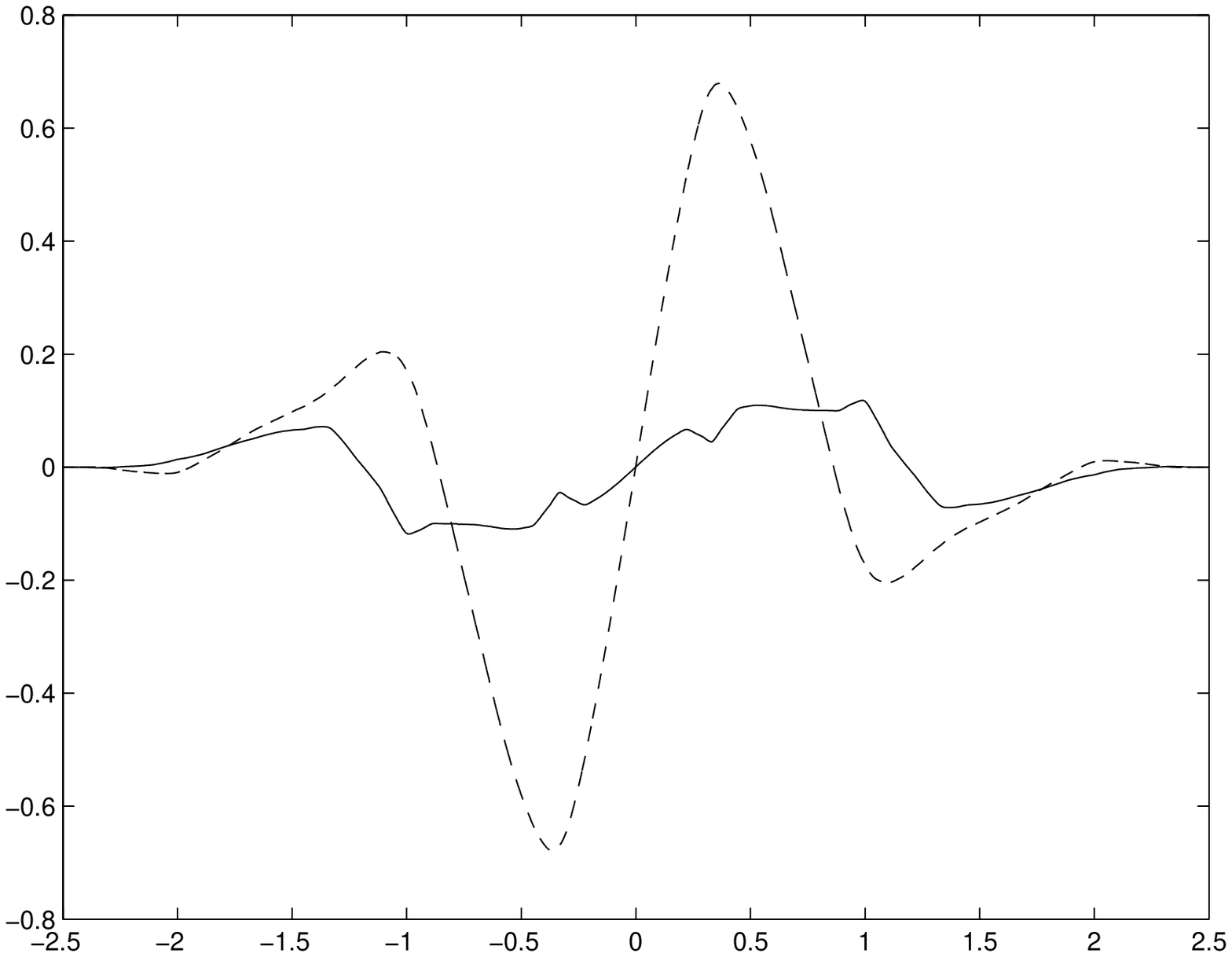,width=1.1in,height=1.0in}
\epsfig{file=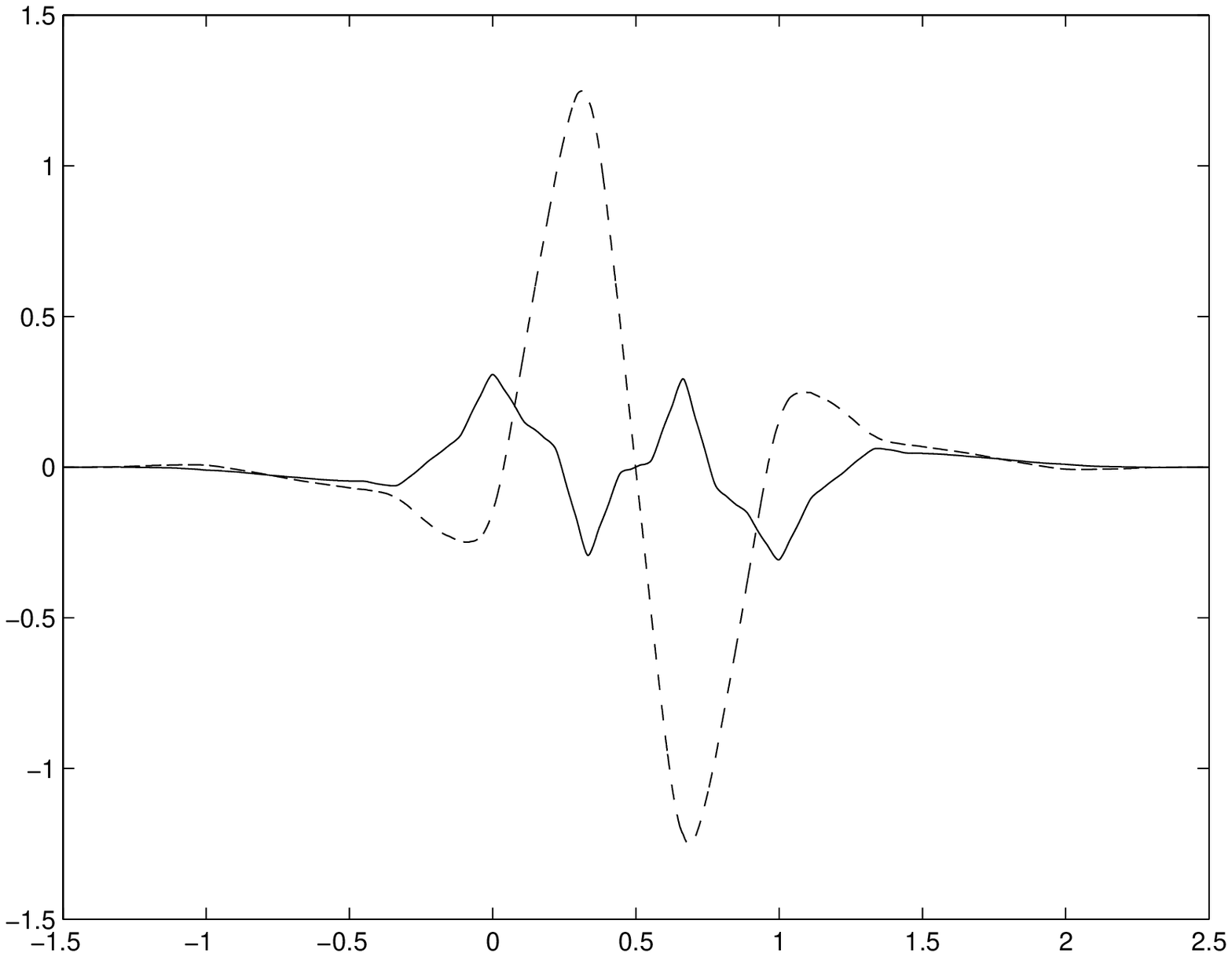,width=1.1in,height=1.0in} }}}
\begin{caption}
{ The graphs of $\phi,\psi^1,\psi^2,\psi^3$ (left to right). Real
part:solid line. Imaginary part:dashed line.}
\end{caption}
\end{figure}
}
\end{example}

\begin{example}{\rm
\label{ex:3} Consider dilation factor $\df=3$. Let  $m=5, n=2$. Then
$P_{5,3}(y) = 1+\frac{40}{3}y+\frac{880}{9}y^2)$. The low-pass
filter $a_0^{5,2}$ with its symbol $\pa_0^{5,2}$  for the complex
pseudo spline of order $(5,2)$ is given by
\[
\pa_0(z)
=\left(\frac{\frac1z+1+z}{3}\right)^5\left[\left(-\frac53-\frac{\sqrt{30}}{3}\iu\right)\frac{1}{z}+\left(\frac{13}{3}+\frac{2\sqrt{30}}{3}\iu\right)+\left(-\frac53-\frac{\sqrt{30}}{3}\iu\right)z\right].
\]
In this case, $m=2n-1$, we have
\[
\begin{aligned}
1-\sum_{\gamma=0}^2\pa_{0;\gamma}(z)^*\pa_{0;\gamma}(z)
&=\frac{1}{19683}(1-z)(1-z^{-1})(4z^3+106z^2-541z
\\&\quad+2320-541z^{-1}+106z^{-2}+4z^{-3}), \quad z=e^{-\iu\xi}.
\end{aligned}
\]
By Riesz lemma, we can factorize the above positive trigonometric
function and obtain $\pa_{0;\df}(z)=c_0(z-r_1)(z^2-r_2 z+r_3)(1-z)$
with $r_1=c_1+\sqrt{c_1^2-1}$, $r_2=c_2+\overline{c_2}$ and
$r_3=|c_2|^2$, where $c_1\in\R, c_2,\overline{c_2}\in\C$ are the
roots for the polynomial $16y^3+212y^2-553y+1054$ and
$c_0=\frac{2\sqrt{3}}{243\sqrt{-r_1r_3}}\in\R$.

Let
$\pp_{\pa_0}=[\pa_{0;0}(z),\pa_{0;1}(z),\pa_{0;2}(z),z^{-2}\frac{1}{\sqrt{2}}\pa_{0;3}(z),z^{-2}\frac{1}{\sqrt{2}}\pa_{0;3}(z^{-1})]$.
Then $\pp_{\pa_0}\pp_{\pa_0}^*=1$. Note that $\pa_{0;1}(z) =z^{-1}
\pa_{0;2}(z^{-1})$. Let $\pp = \pp_{\pa_0} \pU$ with $\pU$ being the
paraunitary matrix given by
\[
\pU:=\left[
       \begin{array}{ccccc}
         1 & 0 & 0 & 0 &0 \\
         0 & \frac{1}{\sqrt{2}} & \frac{1}{\sqrt{2}} & 0& 0 \\
         0 & \frac{1}{\sqrt{2}} & -\frac{1}{\sqrt{2}} & 0 & 0\\
         0&   0&   0 & \frac{1}{\sqrt{2}} & \frac{1}{\sqrt{2}} \\
          0&   0& 0 & \frac{1}{\sqrt{2}} & -\frac{1}{\sqrt{2}}\\
       \end{array}
     \right].
\]
Then $\pp$ is  a $1\times 5$ vector of Laurent polynomials with
symmetry pattern satisfying $\sym\pp=[1,z^{-1},-z^{-1},1,-1]$.
Applying Algorithm~\ref{alg:2}, we can obtain a $5\times 5$
extension matrix $\pP_e$. The coefficient support of $\pP_e$
satisfies $\cs([\pP_e]_{:,j})\subseteq\cs([\pp_{\pa_0}]_j)$ for
$j=1,\ldots,5$. From $\pP_e$, we can derive high-pass filters
$\pa_1,\ldots,\pa_4$  via \eqref{highpass:def} as follows:
\[
\begin{aligned}
 \pa_1(z) &=b_1(z)+b_1(z^{-1}); & \pa_2(z)& =b_2(z)-b_2(z^{-1});
\\
\pa_3(z)& =b_3(z)+z^3b_3(z^{-1});& \pa_4(z)& =b_4(z)-z^3b_4(z^{-1});
\end{aligned}
\]
where
\[
\begin{aligned}
b_1(z)&\approx- 0.38736+ 0.26298\iu- \left(  0.00307+ 0.00204\iu
\right) {z}^{6} - \left(  0.00744+ 0.00660\iu \right) {z}^{5} \\&-
\left(
 0.00951+ 0.01459\iu \right) {z}^{4}+ \left(  0.01372-
 0.02716\iu \right) {z}^{3}
 \\&+ \left(  0.06342- 0.03967\iu
 \right) {z}^{2}+ \left(  0.13657- 0.04144\iu \right) z;
 \\
b_2(z)& \approx\left(  0.00601+ 0.00672\iu \right) {z}^{6}+ \left(
0.01353+ 0.02017\iu \right) {z}^{5} + \left(
 0.01353+ 0.04034\iu \right) {z}^{4}\\&- \left(  0.01370-
 0.03024\iu \right) {z}^{3}
 \\&- \left(  0.07035+ 0.00011\iu
 \right) {z}^{2}- \left(  0.14727+ 0.03064\iu \right) z;
\\
b_3(z)& \approx- \left(  0.01063+ 0.01189\iu \right) {z}^{6}- \left(
0.02392+ 0.03567\iu \right) {z}^{5}- \left(
 0.02392+ 0.07133\iu \right) {z}^{4}
 \\&+ \left(  0.05891-
 0.08433\iu \right) {z}^{3}- \left(  0.00043- 0.20320\iu
 \right) {z}^{2};
\\
b_4(z)&\approx\left(  0.00417+ 0.00467\iu \right) {z}^{6}+ \left(
0.00939+ 0.01400\iu \right) {z}^{5}+ \left(
 0.00939+ 0.02800\iu \right) {z}^{4}
 \\&- \left(  0.02126-
 0.02655\iu \right) {z}^{3}- \left(  0.23403+ 0.31140\iu
 \right) {z}^{2}.
\end{aligned}
\]
The coefficients of $b_1(z),\ldots,b_4(z)$ are rounded from exact
explicit solutions (too long to be presented here). We have
$\pa_1(z)=\pa_1(z^{-1})$, $\pa_2(z)=-\pa_2(1/z)$,
$\pa_3(z)=\pa_3(1/z)$, and $\pa_4(z)=-z^3\pa_4(z^{-1})$. Let $\phi$
be the $3$-refinable function associated with the low-pass filter
$\pa_0$. Let $\psi^1,\ldots,\psi^4$ be the wavelet functions
associated with the high-pass filters $\pa_1,\ldots,\pa_4$
respectively by \eqref{waveletGenerators}. Then $\phi(-\cdot)=\phi$,
$\psi^1(-\cdot)=\psi^1$, $\psi^2(-\cdot)=-\psi^2$,
$\psi^3(1-\cdot)=\psi^3$, and $\psi^4(1-\cdot)=-\psi^4$ . See
Figure~3 for the graphs of $\phi,\psi^1,\psi^2,\psi^3$ and $\psi^4$.
\begin{figure}[th]
\centerline{\scalebox{1.0}{ \hbox{
\epsfig{file=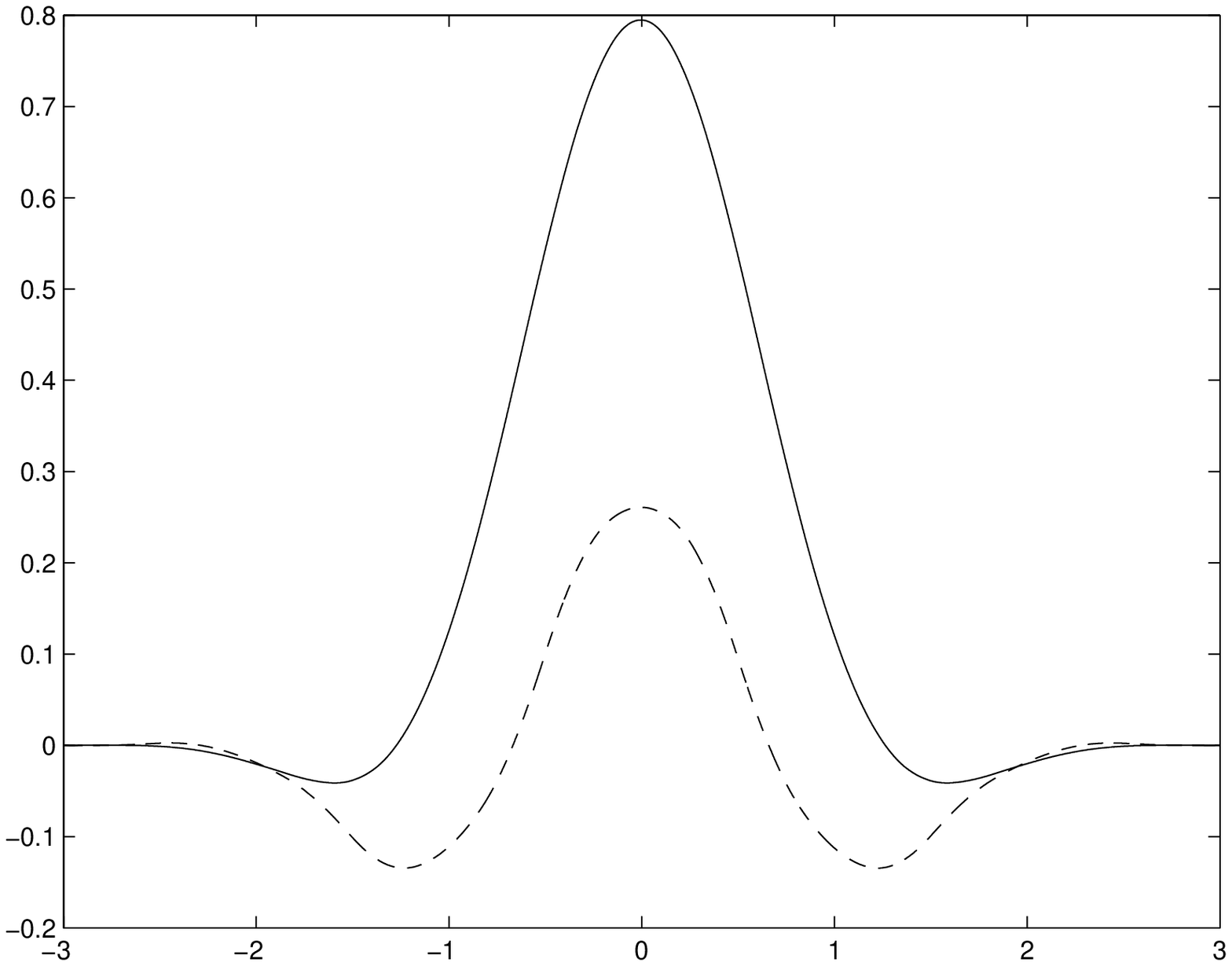,width=1.1in,height=1.0in}
\epsfig{file=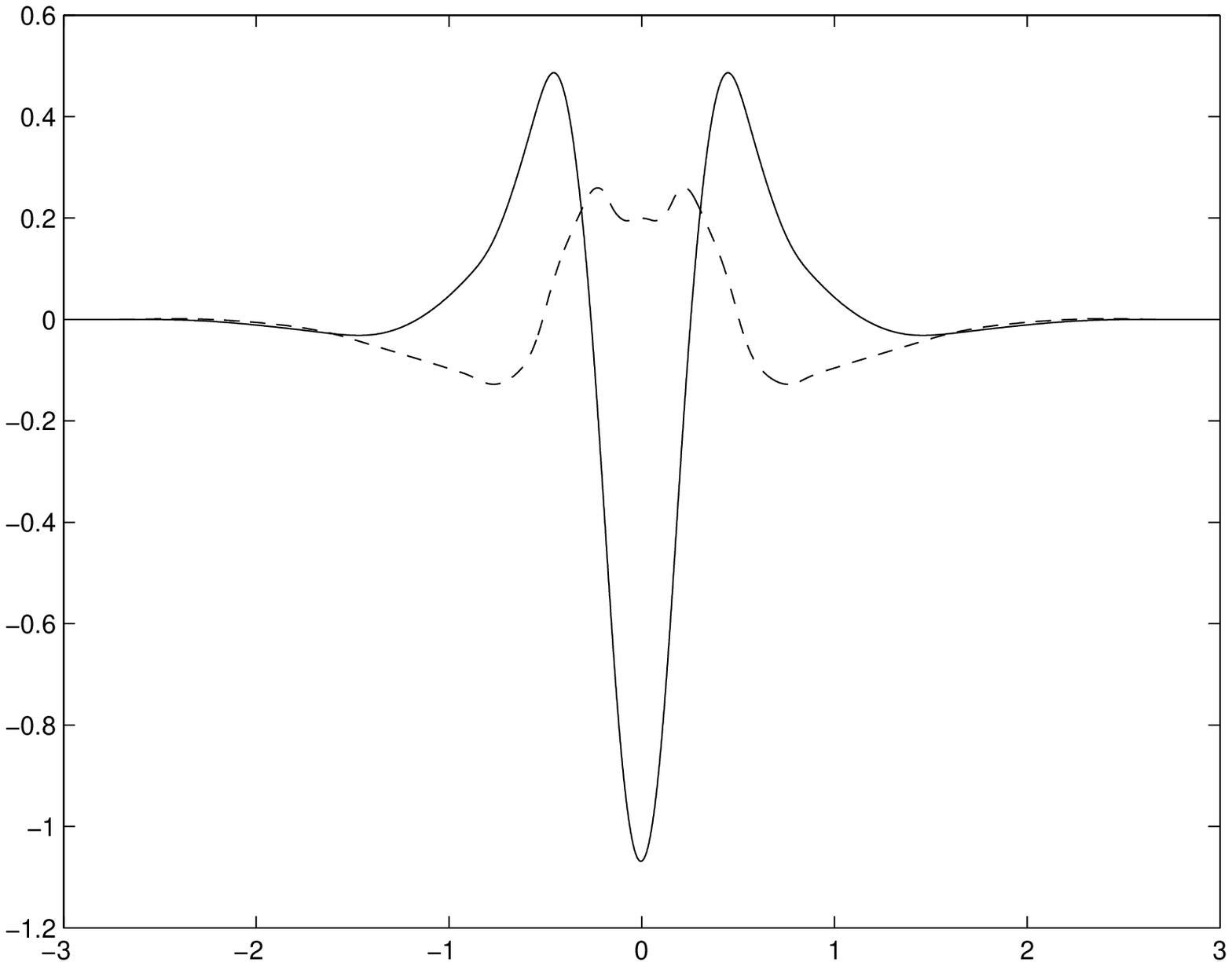,width=1.1in,height=1.0in}
\epsfig{file=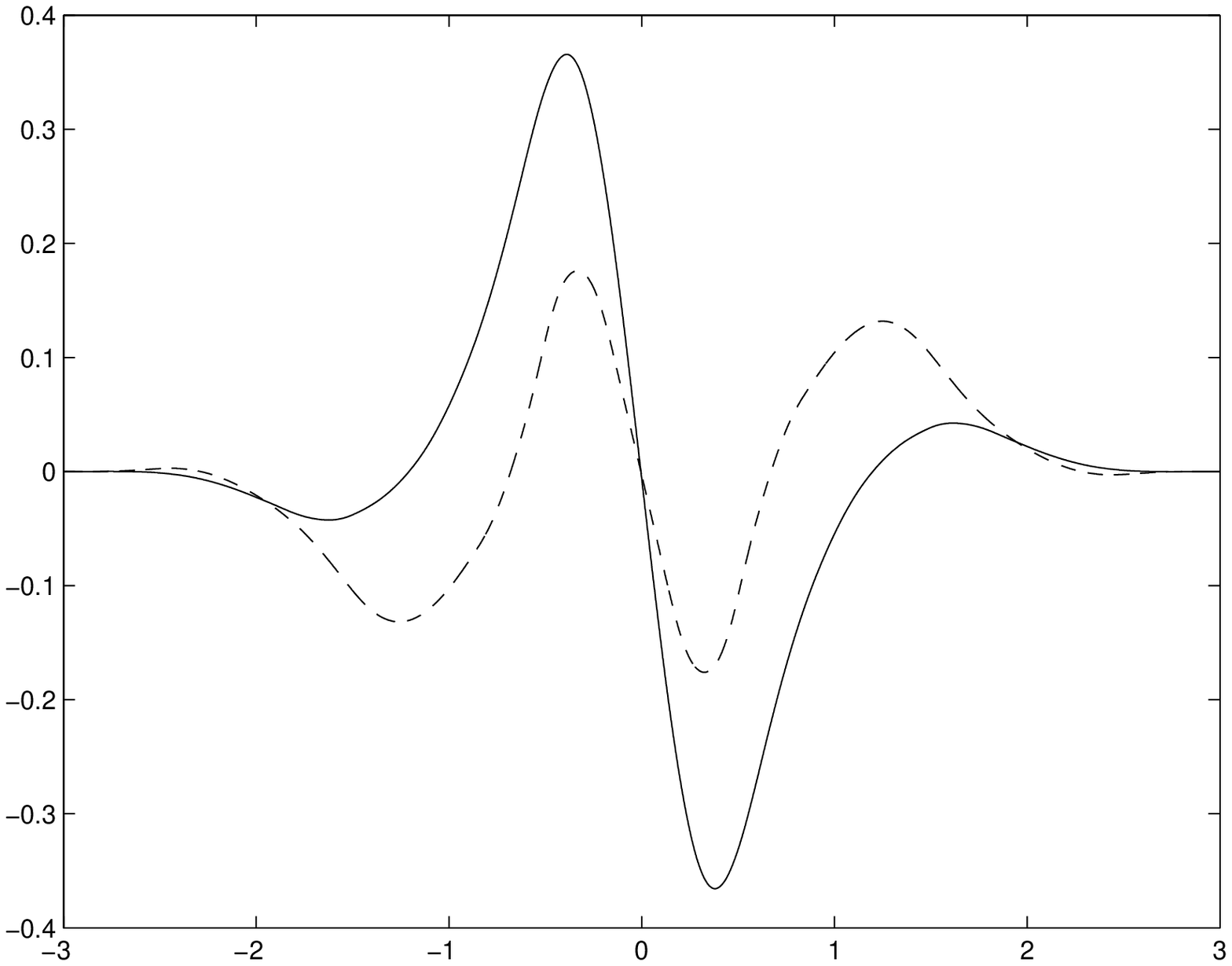,width=1.1in,height=1.0in} } }}
\centerline{\scalebox{1.0}{ \hbox{
\epsfig{file=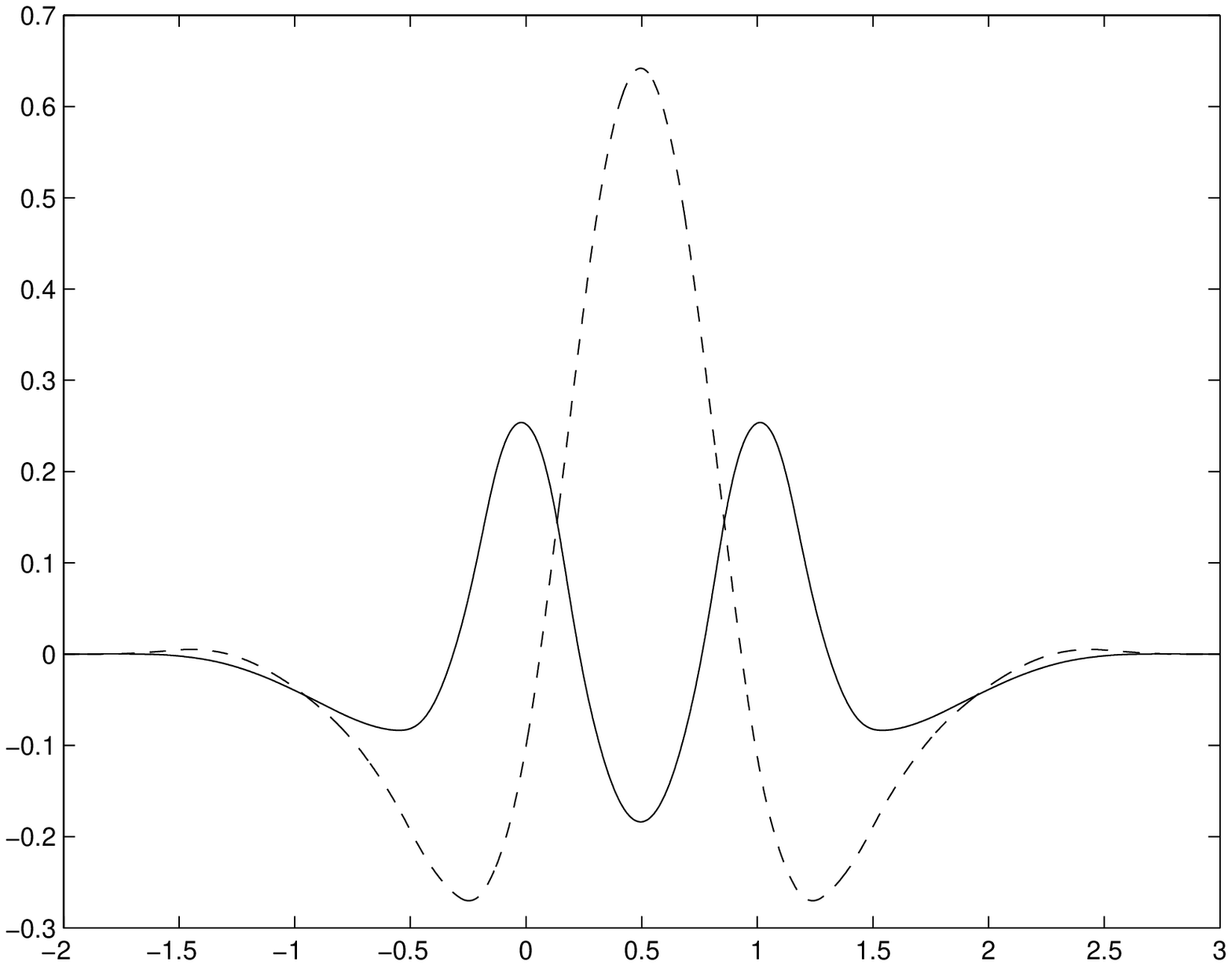,width=1.1in,height=1.0in}
\epsfig{file=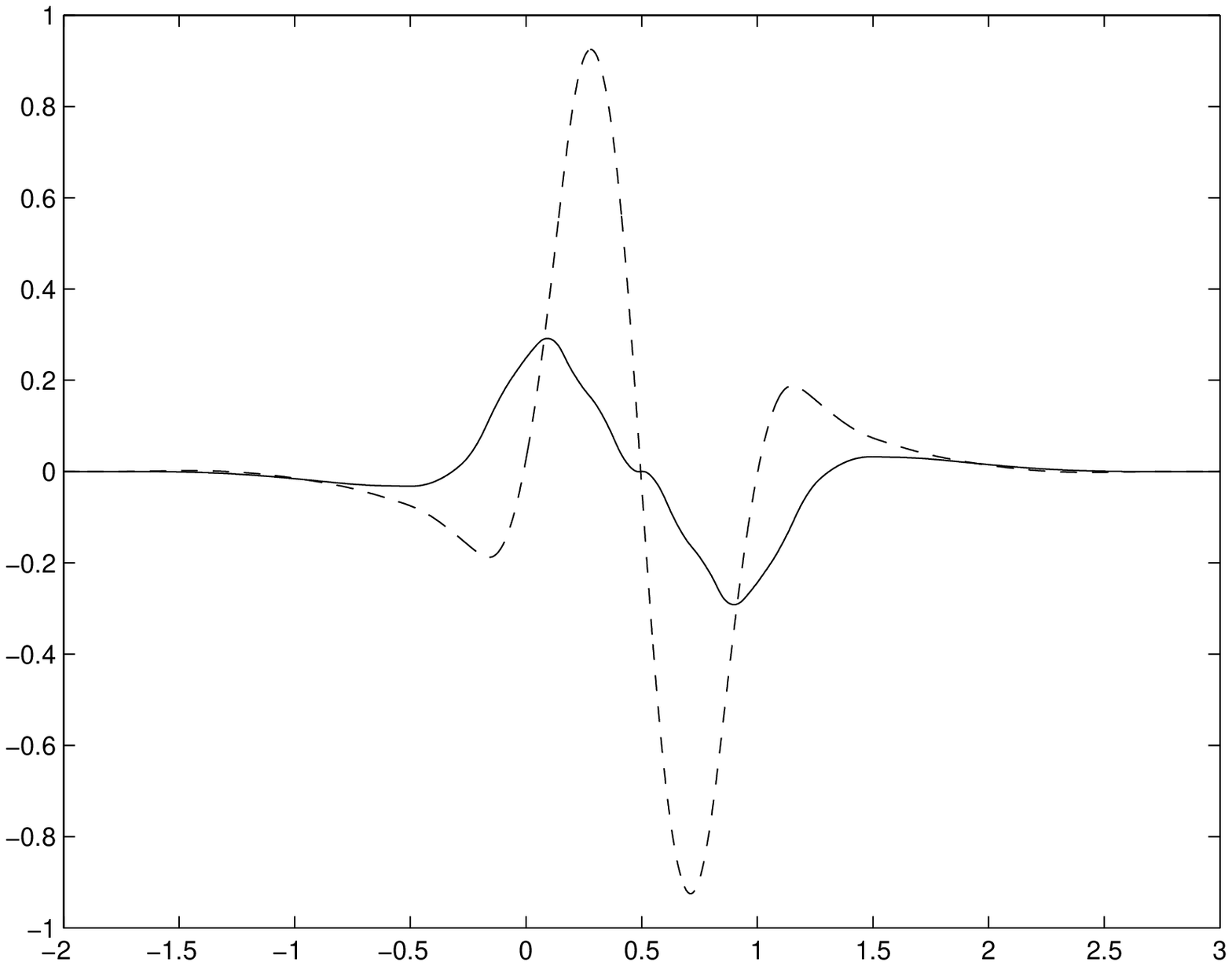,width=1.1in,height=1.0in} } }}
\begin{caption}
{ The graphs of $\phi,\psi^1,\psi^2,\psi^3,\psi^4$ (left to right,
top to bottom). Real part:solid line. Imaginary part:dashed line.}
\end{caption}
\end{figure}

}
\end{example}

\begin{example}{\rm
\label{ex:4} Consider dilation factor $\df=3$. Let  $m=4, n=2$, and
the low-pass filter $a_0^{4,2}$ as in Example~2 with its symbol
$\pa_0$ for the complex pseudo spline $\phi$ of order $(4,2)$ is
given by
\[
\pa_0(z)
=\left(\frac{\frac1z+1+z}{3}\right)^4\left[\left(-\frac43-\frac{2\sqrt{5}}{3}\iu\right)\frac{1}{z}+\left(\frac{11}{3}+\frac{4\sqrt{5}}{3}\iu\right)+\left(-\frac43-\frac{2\sqrt{5}}{3}\iu\right)z\right].
\]
By Theorem~\ref{thm:lin.ind.pseduo.spline}, the shifts of $\phi$ are
linearly independent. Hence, there exist compactly supported  dual
refinable functions in $\wt\phi\in L_2(\R)$  for $\phi$, i.e.,
$\wh{\wt\phi}(\df\xi)=\wt a_0(\xi)\wh{\wt\phi}(\xi)$   for some
low-pass filter $\wt a_0$ and $\langle \phi,\wt\phi(\cdot-k)\rangle
= \delta_k$, $k\in\Z$. Here, we provide a low-pass filter
$\wt\pa_0(z)$ for $\wt\phi$ as follows:
\[
\wt\pa_0(z)=\left(\frac{\frac1z+1+z}{3}\right)^8\left(b(z)+b(\frac1z)\right),
\]
where
\[
\begin{aligned}
b(z)&={\frac {329387}{2754}}+{ \frac {209689}{1377}}\iu\sqrt {5}-
\left({\frac {102661}{816}}+{\frac {5464379}{22032}}\iu\sqrt {5}
\right) z
\\&- \left({\frac {177727}{2754}}-{\frac
{551620}{4131}}\iu \sqrt {5} \right) {z}^{2} + \left( {\frac
{2967467}{22032}}-{\frac { 1034833}{22032}}\iu\sqrt {5} \right)
{z}^{3}
\\&
+ \left( -{\frac {375253}{ 4131}}+{\frac {158555}{15147}}\iu\sqrt
{5} \right) {z}^{4} + \left( { \frac {24620753}{727056}}-{\frac
{29059}{22032}}\iu\sqrt {5} \right) { z}^{5} \\&-{\frac
{24103}{3366}}{z}^{6}+ \left( {\frac {11}{16}}+{\frac
{21391}{727056}}\iu\sqrt {5} \right) {z}^{7}.
\end{aligned}
\]
See Figure~4 for the graph of the $3$-refinable function $\wt\phi$
associated with the low-pass filter $\wt \pa_0(z)$. Note that
$\wt\phi=\wt\phi(-\cdot)$.
\begin{figure}[th]
\centerline{\scalebox{1.0}{ \hbox{
\epsfig{file=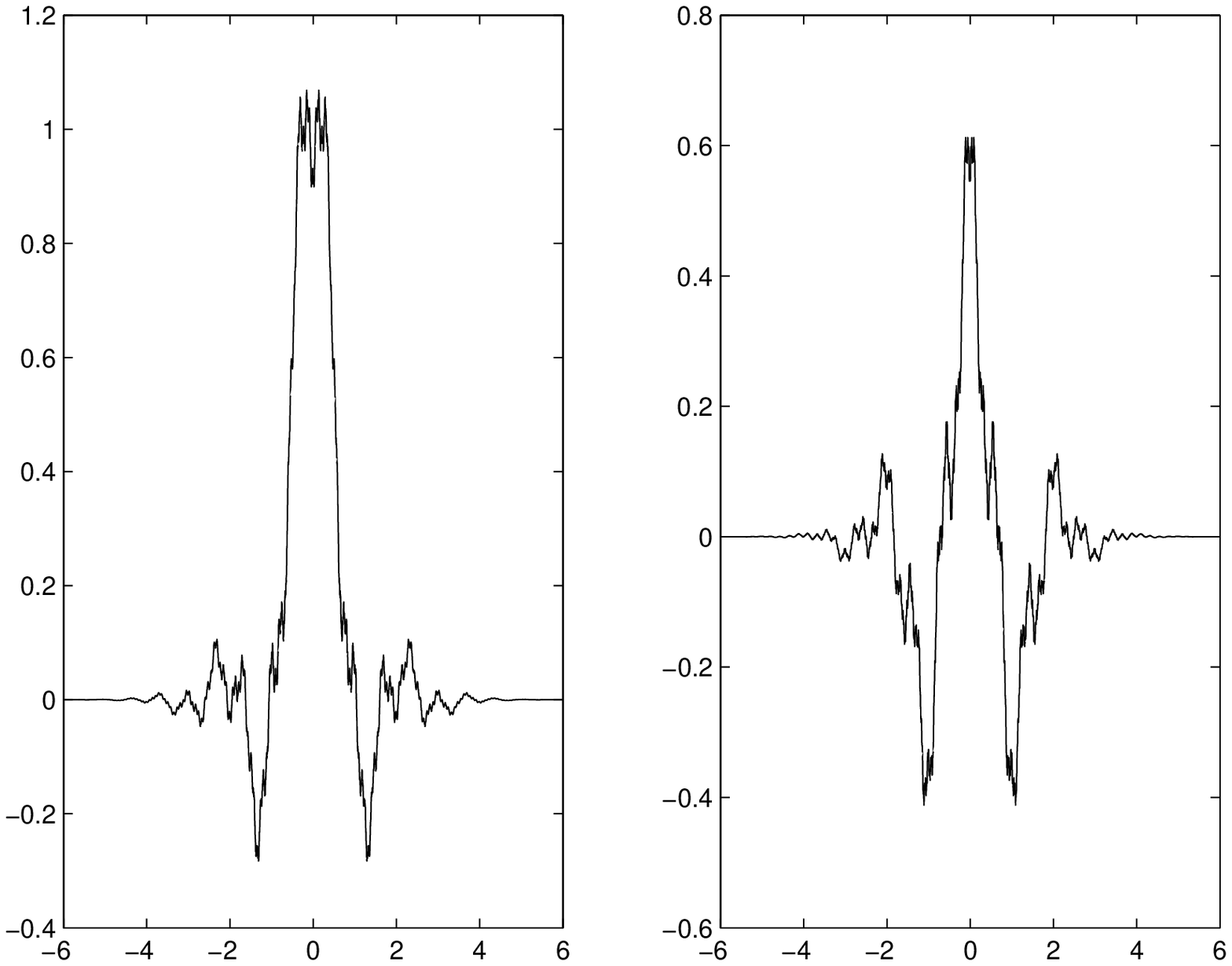,width=3.5in,height=1.3in} }}}
\begin{caption}
{ The graph of $\wt\phi$. Real part:left. Imaginary part:right.}
\end{caption}
\end{figure}
}
\end{example}

\section{Conclusions}

\end{document}